\documentclass[11pt]{amsart}
\usepackage{amsmath}
\usepackage{amssymb, verbatim}
\usepackage{pstricks,pst-node,pst-plot}
\usepackage{comment}
\usepackage{color}
\usepackage{extarrows}
\usepackage{slashed}
\usepackage{esint}

 \title[]{Uniqueness of some cylindrical tangent cones to special Lagrangians}
\author[T. C. Collins]{Tristan C. Collins}
  \email{tristanc@mit.edu}
\thanks{T.C.C is supported in part by NSF CAREER grant DMS-1944952 and an Alfred P. Sloan Fellowship. } 

 \author[Y. Li]{Yang Li}
  \email{yangmit@mit.edu}

  \thanks{Y. L. is a current Clay Research Fellow and a CLE Moore
Instructor at MIT. }
    \address{Department of Mathematics, Massachusetts Institute of Technology, 77 Massachusetts Avenue, Cambridge, MA 02139}
  \dedicatory{Dedicated to H. Blaine Lawson Jr., with admiration, on the occasion of his 80th birthday.}

\theoremstyle{plain}
\newtheorem{thm}{Theorem}[section]
\newtheorem{prop}[thm]{Proposition}
\newtheorem{defn}[thm]{Definition}
\newtheorem{lem}[thm]{Lemma}
\newtheorem{cor}[thm]{Corollary}

\theoremstyle{definition}
\newtheorem{ex}[thm]{Example}
\newtheorem{rk}[thm]{Remark}

\numberwithin{equation}{section}

\newcommand{\del}{\partial}

\newcommand{\dbar}{\overline{\del}}
\newcommand{\ddb}{\sqrt{-1}\del\dbar}

\newcommand{\bC}{\mathbb{C}}
\newcommand{\bR}{\mathbb{R}}

\newcommand{\cM}{\mathcal{M}}
\newcommand{\cH}{\mathcal{H}}
\newcommand{\ux}{\underline{x}}

\newcommand{\be}{\begin{equation}}
\newcommand{\bea}{\begin{eqnarray}}
\newcommand{\eea}{\end{eqnarray}} 
\newcommand{\ee}{\end{equation}}

\newcommand{\bfC}{\mathbf{C}}

\newcommand{\Cone}{{\rm Cone}}
\newcommand{\mfC}{\mathfrak{C}}
\newcommand{\mfH}{\mathfrak{H}}
\newcommand{\cK}{\mathcal{K}}
\newcommand{\uC}{\underline{C}}
\newcommand{\hatN}{\widehat{N}}
\newcommand{\hatmfC}{\widehat{\mathfrak{C}}}
\newcommand{\cB}{\mathcal{B}}
\newcommand{\cF}{\mathcal{F}}
\newcommand{\cV}{\mathcal{V}}

\renewcommand{\leq}{\leqslant}
\renewcommand{\geq}{\geqslant}

\renewcommand{\epsilon}{\varepsilon}
\renewcommand{\phi}{\varphi}

\begin{document}
\maketitle

\begin{abstract}
We show that if an exact special Lagrangian $N\subset \mathbb{C}^n$ has a multiplicity one, cylindrical tangent cone of the form $\mathbb{R}^{k}\times \bfC$ where $\bfC$ a special Lagrangian cone with smooth, connected link, then this tangent cone is unique provided $\bfC$ satisfies an integrability condition.  This applies, for example, when $\bfC= \bfC_{HL}^{m}$ is the Harvey-Lawson $T^{m-1}$ cone for $m\ne 8,9$.
\end{abstract}

\section{introduction}

Let $N\subset \mathbb{R}^{n+k}$ be a codimension $k$ minimal surface, with $0\in N$.  For any sequence $\lambda_i \rightarrow +\infty$ the rescaled surfaces $\lambda_iN$ converge subsequentially to a minimal cone $\mfC$, called a tangent cone of $N$ at $0$.  A fundamental problem in the study of minimal surfaces is to understand if $\mfC$ is unique, or if it depends on the sequence of rescalings.  Foundational results of Allard-Almgren \cite{AA} and Simon \cite{Simon83} establish the uniqueness of the tangent cone assuming at least one tangent cone is smooth and of multiplicity one away from $0$.  Results of this nature have important applications for the regularity theory of minimal surfaces; see \cite{CMLoj, BK, CES, DLSS, White, Taylor, Liang, SimonRec, SimonRec2, NV, NV2, PRiv, Sz2} and the references therein for related work. 

For tangent cones with non-isolated singularities the simplest example is that of a {\em cylindrical tangent cone} $\mfC = \mathbb{R}^k \times \bfC$, where $\bfC$ is a minimal cone with smooth cross-section.  In the case of codimension $1$ minimal hypersurfaces, Simon \cite{SimonUn} proves that such tangent cones are unique under some assumptions on the Jacobi fields normal to $\bfC$.  Recent work of Sz\'ekelyhidi \cite{Sz} addresses the uniqueness in some important cases where Simon's result does not apply, including when $\bfC$ is the Simons cone in $\mathbb{R}^8$.  The goal of the present work is to address the uniqueness of cylindrical tangent cones for a natural class of higher codimensional volume minimizers called {\em special Lagrangians}.

Let $(X,\omega,\Omega)$ be a Calabi-Yau manifold of real dimension $2n$, with symplectic form $\omega$ and holomorphic $(n,0)$-form $\Omega$.  The most basic example of a Calabi-Yau manifold, and the one relevant for our purposes, is $\mathbb{C}^n$ equipped with holomorphic coordinates $z_i= x_i + \sqrt{-1}y_i$, symplectic form $\omega= \sum_i dx_i\wedge dy_i$ and holomorphic volume form $\Omega = dz_1\wedge \cdots \wedge dz_n$.  In their landmark paper on calibrated geometries Harvey-Lawson \cite{HL} introduced the notion of a special Lagrangian, which is a $n$-dimensional submanifold $N\subset X$ satisfying
\[
\omega\big|_{N}=0, \qquad {\rm Im}(\Omega)\big|_{N}=0.
\]
Harvey-Lawson showed that special Lagrangian submanifolds are automatically volume minimizing in their homology class, and these manifolds now play a distinguished role in many aspects of Calabi-Yau geometry, particularly in mirror symmetry.  

Our first main theorem establishes the uniqueness of cylindrical tangent cones to special Lagrangians under an integrability assumption.

\begin{thm}\label{thm: main1}
Suppose $N\subset \mathbb{C}^n$ is a multiplicity $1$, closed integral current which is special Lagrangian,  and that $0 \in {\rm supp} N$. Suppose that $N$ is exact (see Definition~\ref{defn: exact}) and some tangent cone of $N$ at $0$ is of the form $\mfC= \mathbb{R}^k \times \bfC$ where $\bfC$ is a special Lagrangian cone with smooth, connected link and $\dim \bfC>2$.  Suppose in addition that $\bfC$ is integrable in the sense of Definition~\ref{defn: strongInt}. Then $\mfC$ is the unique tangent cone of $N$ at $0$.  
\end{thm}
\begin{rk}
The proof of Theorem~\ref{thm: main1} also yields an explicit (but likely not optimal) estimate for the rate of convergence;  see~\eqref{eq: main1RateEst}.
\end{rk}

\noindent Roughly speaking, the special Lagrangian cone $\bfC$ is {\em integrable} if every infinitesimal deformation of $\bfC$ corresponding to a harmonic function with linear or quadratic growth integrates to a genuine deformation through special Lagrangian cones.  For example, $\bfC$ is integrable if $\bfC$ has the stronger property that every harmonic function on $\bfC$ of linear or quadratic growth is generated by the action of the automorphism group $SU(n-k)\ltimes \mathbb{C}^{n-k}$; we call such special Lagrangians {\em rigid}.  For experts in minimals surfaces, rigidity in the sense of this paper is analogous to conditions $\ddagger(a)$ and $\ddagger(b)$ in \cite[Page 4]{SimonUn}; in particular, we do not require any condition analogous to $\ddagger(c)$.

  An important example of rigid special Lagrangian cones are the Harvey-Lawson cones \cite{HL}
\[
\bfC^m_{HL}:= \bigg\{(z_1,\ldots,z_m) \in \mathbb{C}^m : |z_1|=\cdots =|z_m|, \quad {\rm Arg}(i^{m+1}z_1\cdots z_m)=0\bigg\}.
\]
for $m\ne 8,9$.  Thus, Theorem~\ref{thm: main1} yields

\begin{cor}\label{cor: main1}
Suppose $N\subset \mathbb{C}^{k+m}$ is a multiplicity $1$, closed integral current which is special Lagrangian,  and that $0 \in {\rm supp} N$.  Suppose that $N$ is exact and that some tangent cone of $N$ at $0$ is of the form $\mfC= \mathbb{R}^k \times \bfC_{HL}^{m}$. If $m\ne 8,9$ then $\mfC$ is the unique tangent cone of $N$ at $0$.
\end{cor}

\noindent Note that when $m=3$ the work of Haskins shows that the cone $\bfC_{HL}^{3}$ is the unique {\em strictly stable} special Lagrangian $T^2$ cone \cite{Haskins}.

Finally, we remark that when $\dim \bfC \geq 5$ we obtain polynomial convergence to the tangent cone.  Recall that for compact sets $K_1, K_2$ and a bounded set $E$ we denote the Hausdorff distance in $E$ to be
\[
d^{H}(K_1, K_2;E) =  \sup\{{\rm dist}(x,y) : x\in K_1\cap \overline{E},\,\, y\in K_{2}\cap \overline{E}\}.
\]

\begin{thm}\label{thm: main2}
Under the assumptions of Theorem~\ref{thm: main1}, assume in addition that $\dim \bfC \geq 5$. Then there constants $C, \alpha >0$ such that
\[
d^{H}(\rho^{-1}N, \mfC; B_{1}) \leq C\rho^{\alpha}.
\]
for all $\rho$ sufficiently small.
\end{thm}

\begin{rk}
It may be possible to extend the rate estimate of Theorem~\ref{thm: main2} to the case when $\dim \bfC =4$ by modifying the proof of Proposition~\ref{prop: decayOfVolEx}.  The case of $\dim \bfC=3$ seems to require a different approach.
\end{rk}

\begin{rk}
Let us also remark on the case of tangent cones of the form $\mathbb{R}^k \times \bfC$ for $\dim \bfC=2$.  By hyperK\"ahler rotation, one can show that any special Lagrangian cone in $\mathbb{C}^2$ is necessarily a union of special Lagrangian planes intersecting at the origin (possibly with multiplicity).  Thus, if $\bfC$ is assumed to have connected link and $\mfC$ has multiplicity one, then in fact $\mfC$ is smooth and the uniqueness of the tangent plane follows from \cite{AA,Simon83}.  
\end{rk}

Broadly speaking, the general strategy goes back to work of Allard-Almgren \cite{AA} and Simon \cite{Simon83, SimonUn}. We would like to view $N$ as a small perturbation of the tangent cone $\mfC$ which is controlled by the linearized special Lagrangian graph equation (ie. the Laplace equation) on $\mfC$.  The main difficulty is that $\mfC$ does not have isolated singularities and so the linearized problem may not accurately approximate $N$.  One needs to prevent the deviation of $N$ from $\mfC$ concentrating near the singular set  $\mathbb{R}^k\times \{0\} \subset \mfC$.  In Simon's work \cite{SimonUn} this is overcome by constructing appropriate comparison surfaces using the work of Hardt-Simon \cite{HS}.  Sz\'ekelyhidi \cite{Sz} employs a similar, but necessarily more elaborate argument exploiting a discrete replacement of the \L ojasiewicz inequality.

Let us give a heuristic overview of the current paper, which is inspired by Simon's paper \cite{SimonUn}.  Fix coordinates $(z,z') \in \mathbb{C}^k \times \mathbb{C}^{n-k}$ and write $z_i = x_i +\sqrt{-1}y_i$.  Assume that some tangent cone of $N$ at $0$ is of the form $\mfC= \mathbb{R}^k \times \bfC$ where $\mathbb{R}^k = \{y_i=0: 1\leq i \leq k\}$.  The {\em normalized volume excess} is defined to be
\begin{equation}\label{eq: volExDefn}
\begin{aligned}
{\rm VolEx}_{N}(r) &= r^{-n}\mathcal{H}^n(N\cap B_{r}) - \left(\lim_{s\rightarrow 0^+} s^{-n}\mathcal{H}^n(N\cap B_{s})\right)\\
&= r^{-n}\mathcal{H}^n(N\cap B_{r}) - r^{-n} \mathcal{H}^n(\mfC \cap B_{r})
\end{aligned}
\end{equation}
where $\mathcal{H}^{n}$ denotes the $n$-dimensional Hausdorff measure.  By volume monotonicity for minimal surfaces, ${\rm VolEx}_{N}(r)$ is an increasing function of $r$. 

Let us consider the simplified model in which $N$ is a graph over $\mfC= \mathbb{R}^k\times \bfC$ of a $1$-form $df$ in suitable Darboux coordinates.   Here $f: \mfC \rightarrow \mathbb{R}$ is a function solving a uniformly elliptic PDE which is a small perturbation of the Laplace equation.  Since some tangent cone to $N$ at $0$ is $\mfC$, we may assume that $f$ is $C^{\infty}$ small in $\mfC \cap B_{2} \cap \{|z'|>2\tau\}$ for some fixed, small $\tau>0$.  The assumption that $N$ is exact means that there is a function $\beta:N\rightarrow \mathbb{R}$ (say normalized to have $\beta(0)=0$) such that $d\beta=\frac{1}{2}\lambda$ where $\lambda$ is the Liouville form on $\mathbb{C}^n$.    One key observation is that the functions $\beta, y_i$ are harmonic functions on $N$, which can be expressed in terms of $f$ as
\begin{equation}\label{eq: introBYrel}
\beta= -\frac{1}{2}R^2 \frac{\del}{\del R} \left(\frac{f}{R^2}\right) \qquad y_i = -\frac{\del f}{\del x_i}
\end{equation}
where $R= \sqrt{|z|^2+|z'|^2}$ is the radial function on $\mathbb{C}^n$.  Denote by $\| \cdot \|_{L^2}$ suitably defined scale invariant $L^2$ norms (see ~\eqref{eq: normDefn} for precise formulas).  By first integrating the formula for $\beta, y$ we obtain a bound 
\[
|f| < C\left(\|\beta\|_{L^2(N\cap B_4)} + \|y\|_{L^2(N\cap B_4)} + |z'|^2  \|f\|_{L^2(B_2\cap \{|z'|>2\tau\})}\right);
\]
see Proposition~\ref{prop: propOfSmall}. Thus, if 
\begin{equation}\label{eq: case1Intro}
\|\beta\|_{L^2(N\cap B_4)} + \|y\|_{L^2(N \cap B_4)} \leq C\|f\|_{L^2(B_2\cap \{|z'|>2\tau\})},
\end{equation}
 then $f$ is controlled on smaller scales by its $L^2$ norm on $\mfC \cap B_{2} \cap \{|z'|>2\tau\}$; this is a non-concentration type result which implies that $N$ is well controlled by the Laplace equation on $\mfC$.  
 
 The main decay result, Proposition~\ref{prop: fDecay} is obtained by using a blow-up argument to reduce to the spectral properties of $\mfC$.  Assuming $f$ is harmonic on $\mfC$, we decompose $f$ into a sum of homogeneous harmonic functions.  If the volume excess satisfies
\begin{equation}\label{eq: case2Intro}
{\rm VolEx}_{N}(2) \leq C\|f\|^2_{L^2(B_2\cap \{|z'|>2\tau\})}
\end{equation}
then the expansion of $f$ can only contain terms with homogeneous degree at least $2$.  Roughly speaking, the assumption of integrability means that the degree $2$ terms in the expansion of $f$ can be removed by modifying the model cylinder $\mfC \rightarrow \mfC'$, which is a perturbation of size $\sim \|f\|_{L^2(B_2\cap \{|z'|>2\tau\})}$.  Thus, we may assume that $f$ contains only terms of degree strictly larger than $2$.  But this implies that for some $0<\lambda<1$ we have
\begin{equation}\label{eq: introPLDecayf}
\|f\|_{L^2(B_{\frac{1}{2}}\cap \{|z'|>\tau\})} \leq \lambda\|f\|_{L^2(B_2\cap \{|z'|>2\tau\})}.
\end{equation}
If such an estimate holds at all scales, then an iteration easily implies the polynomial convergence of $\rho^{-1} N$ to $\mfC$, its unique tangent cone, as $\rho \rightarrow 0$.  Thus, we only need to address the possibility that either of~\eqref{eq: case1Intro} or ~\eqref{eq: case2Intro} fail. If~\eqref{eq: case1Intro} fails, then the elliptic regularity theory applied to $f$, together with~\eqref{eq: introBYrel}, implies a fast decay property
\[
\|\beta\|_{L^2(N\cap B_1)} + \|y\|_{L^2(N \cap B_1)} \leq \frac{1}{100} \left(\|\beta\|_{L^2(N\cap B_4)} + \|y\|_{L^2(N \cap B_4)} \right)
\]
provided $C$ is sufficiently large; see Lemma~\ref{lem: easyDecay}.  By assumption, $\beta, y$ control $f$ in this case.  If~\eqref{eq: case2Intro} fails, and we assume $\dim \bfC\geq 5$, we show in Proposition~\ref{prop: decayOfVolEx} that the volume excess decays;
\[
{\rm VolEx}_{N}\left(\frac{1}{2}\right) \leq \frac{1}{2} {\rm VolEx}_{N}(2). 
\]
Thus, $f$ is again controlled by a fast decaying quantity.  Applying this trichotomy iteratively yields Theorem~\ref{thm: main2}.

The main technical difficulties in making this heuristic argument rigorous are the following;
\begin{itemize}
\item[(i)] The potential function $f$ is only locally defined, and in general, is not defined near the singular set of $\mfC$.  Thus, the above iteration needs to include the statement that the potential continues to exist at smaller scales.  The necessary quantitative extension result for the potential is proved in Proposition~\ref{prop: propOfSmall}.
\item[(ii)] In the iteration we need quantitative control of the drift of the model cylinders $\mfC \rightarrow \mfC_1 \rightarrow \mfC_2 \cdots$.  This amounts to showing the summability of $\|f\|_{L^{2}}$ over dyadic scales.  When $\dim \bfC\geq 5$, the summability is not problematic since we obtain power law decay in all cases. However, when $\dim \bfC \leq 4$, we are not able to prove the power law decay of the volume excess when ~\eqref{eq: case2Intro} fails.  In this case the estimate
\[
\|f\|^2_{L^2(B_2\cap \{|z'|>2\tau\})} \leq C^{-1}{\rm VolEx}_{N}(2)
\]
is not necessarily summable over all scales (since small scales are counted many times).  We circumvent this issue by replacing ~\eqref{eq: case2Intro} with an effective version of the form
\[
{\rm VolEx}_{N}(2)-{\rm VolEx}_{N}(2 \cdot 2^{-b}) \leq C\|f\|^2_{L^2(B_2\cap \{|z'|>2\tau\})}
\]
for some fixed, large $b$.  This necessitates modifying~\eqref{eq: introPLDecayf} since we can no longer rule out the presence of homogeneous harmonic functions with degree less than $2$ in the expansion of $f$.
\end{itemize}

Let us now describe the outline of the paper.  Section~\ref{sec: basic} collects many of the basic facts we will need throughout the paper, including the deformation theory of special Lagrangian cones, the role of various canonical harmonic functions, and the existence of adapted Darboux coordinate systems.  Section~\ref{sec: basic} also contains the proof that the Harvey-Lawson cones $\bfC_{HL}^{m}$ are integrable (in fact, rigid) when $m\ne 8,9$.   Section~\ref{sec: pot} proves various results concerning the existence of local potentials for $N$, as a graph over $\mfC$, leading to the quantitative existence/non-concentration result Proposition~\ref{prop: propOfSmall}.  Finally, Section~\ref{sec: Decay} proves various results concerning the decay of scale invariant norms.  The main result of this section, Proposition~\ref{prop: fDecay}, establishes the decay of the potential when concentration can be ruled out.  Finally, at the end of Section~\ref{sec: Decay} we combine our results to prove Theorem~\ref{thm: main1} and Theorem~\ref{thm: main2}.  We have chosen to give a detailed proof of the more general (and more complicated) case contained in Theorem~\ref{thm: main1} and sketch the proof of Theorem~\ref{thm: main2}, which follows largely the same lines.  
\\
\\
{\bf Acknowledgements:} The authors are grateful to G\'abor Sz\'ekelyhidi for helpful comments on an earlier draft.

\subsection{Notation}
Throughout this paper we will use the following notation:
\begin{itemize}
\item $N$ is a closed integral current of multiplicity $1$, which is special Lagrangian, and hence area minimizing.  
\item We write  $\mathbb{C}^n = \mathbb{C}^k \times \mathbb{C}^{n-k}$.  We take standard complex coordinates $(z, z') \in \mathbb{C}^k \times \mathbb{C}^{n-k}$, and write $z_i = x_i + \sqrt{-1}y_i$ for $1\leq i \leq k$, and $z'_i = x'_i +\sqrt{-1}y'_i$ for $1 \leq i \leq n-k$.  In these coordinates the symplectic form is given by 
\[
\omega_{std} = \frac{\sqrt{-1}}{2} \sum_{i=1}^k dz_i\wedge d\bar{z}_{i} + \frac{\sqrt{-1}}{2}\sum_{j=1}^{n-k}dz'_j \wedge d\bar{z'}_j
\]
\item $B_{\rho}(p)$ will denote the ball of radius $\rho$ centered at $p$ in $\mathbb{C}^n$.  We will also write $B_{\rho}$ for the ball centered at $0 \in \mathbb{C}^n$.
\item $\mathfrak{C}$ will denote a special Lagrangian cylinder of the form 
\[
\{y_1=y_2=\cdots=y_k=0\} \times \bfC\simeq \mathbb{R}^k\times \bfC \subset \mathbb{C}^k\times \mathbb{C}^{n-k}
\]
 where $\bfC\subset \mathbb{C}^{n-k}$ is a special Lagrangian cone with smooth, connected link.  We will denote by $\Sigma$ the link of $\bfC$, and occasionally write 
 \[
 \bfC= {\rm Cone}(\Sigma)
 \]
 to emphasize that $\bfC$ is the cone over $\Sigma$.   We will always assume that $n-k >2$.  
 
 \item  Let $\Omega_{n-k}=dz_i'\wedge \cdots \wedge dz_{n-k}'$ be the holomorphic volume form on $\bC^{n-k}$ normalized so that ${\rm Im}(\Omega_{n-k})\big|_{\bfC}=0$ and let $\Omega_k = dz_1\wedge \ldots \wedge dz_k$ be the holomorphic volume form on $\bC^k$. Identify $\bR^k =\{y_j=0: 1 \leq j \leq k\}$ denote the real $k$-plane.  Then it is easy to check that $\mathfrak{C}$ is special Lagrangian in $\mathbb{C}^n$ for the volume form $\Omega:=  \Omega_k\wedge \Omega_{n-k}$.  In our normalization
\[
\frac{\omega^{n}}{n!} = (-1)^{\frac{n(n-1)}{2}} (\frac{\sqrt{-1}}{2})^{n} \Omega \wedge \overline{\Omega}.
\]

\item The Liouville form will be denote by
\[
\lambda = \sum_{i=1}^{k} x_idy_i -y_idx_i + \sum_{j=1}^{n-k} x'_jdy'_j-y'_jdx'_j
\]
where we note that $d\lambda = 2\omega$. 

\item We denote by $r = |z'|$ the radial function in $\mathbb{C}^{n-k}$, and by $R = \sqrt{|z|^2+|z'|^2}$ the radial function in $\mathbb{C}^n$.  We also denote by $r\frac{\del}{\del r}, R\frac{\del}{\del R}$ the radial vector fields on $\mathbb{C}^{n-k} \subset \mathbb{C}^n$ and $\mathbb{C}^n$ respectively.  

\item All integrals, unless otherwise noted, are with respect to the $n$-dimensional Hausdorff measure, which we denote by $\mathcal{H}^n$. 

\item If $\bfC= {\rm Cone}(\Sigma)$ is a special Lagrangian cone with link $\Sigma$ a connected special Legendrian in $S^{2m-1}$, we will denote by $\mathcal{M}$ the moduli space of special Legendrian deformations of $\Sigma$.  Throughout the paper $\mathcal{M}$ will be a smooth manifold and we will denote by $\exp^{\mathcal{M}}$ the exponential map induced by the natural $L^2$ metric on $\mathcal{M}$ and $\mathcal{K}$ will denote a connected compact subset of $\mathcal{M}$. 

\item Throughout the paper we will use $C$ to denote a non-negative constant, which can increase from line to line, but always depends only on the stated quantities.

\end{itemize}

\section{Basic Results}\label{sec: basic}

\subsection{Harmonic functions on special Lagrangian cones}

A well-known result of McLean \cite{McLean} says that infinitesimal deformations of a smooth, compact special Lagrangian $N$ correspond to harmonic $1$-forms on $N$.  Joyce \cite{JoyceII} proved a  vast generalization of this result for special Lagrangians with isolated conical singularities.  If $\bfC\subset \mathbb{C}^m$ is a special Lagrangian cone with an isolated singularity at $0$, then an important role in the deformation theory for $\bfC$ is played by harmonic functions.  Recall that if $\bfC = \Cone(\Sigma)\subset \mathbb{C}^m$ is special Lagrangian then the link  $\Sigma$ is a special Legendrian submanifold of $S^{2m-1}$ and conversely.  We have the following result \cite{FHY, Moriyama};

\begin{lem}
The infinitesimal deformation space of a special Legendrian submanifold $\Sigma \subset S^{2m-1}$ is isomorphic to the space of functions $\phi:\Sigma\rightarrow \mathbb{R}$ satisfying $\Delta_{\Sigma}\phi = 2m \phi$.
\end{lem}

Eigenfunctions of the Laplacian on $\Sigma$ give rise to harmonic functions on the cone $\bfC= \Cone(\Sigma)$ by the usual separation of variables construction.  Indeed, if $v= r^{\alpha}\phi(\omega)$ is a homogeneous function of order $\alpha$, then
\begin{equation}\label{eq: sepVarHar}
\Delta_{\bfC} v = r^{\alpha-2}\left( \Delta_{\Sigma} \phi - \alpha(\alpha+m-2)\phi\right).
\end{equation}
In particular, the infinitesimal deformations of the special Legendrian link $\Sigma$ correspond exactly to quadratic growth harmonic functions on $\bfC$. In general this deformation problem is obstructed  \cite{Moriyama}.

\begin{defn}\label{defn: strongInt}
Let $\bfC = \Cone(\Sigma) \subset \mathbb{C}^m$ be a special Lagrangian cone with smooth, connected special Legendrian link $\Sigma$.  Let 
\[
\mathfrak{H}_{\alpha}(\Sigma):= \{\phi:\Sigma\rightarrow \mathbb{R}: \Delta_{\Sigma}\phi =\alpha \phi \}.
\]
We say that $\bfC$ is {\bf integrable} if
\begin{itemize}
\item[$(i)$] $\dim \mathfrak{H}_{m-1}(\Sigma) = 2m$, and
\item[$(ii)$] if $\dim \mathfrak{H}_{2m}(\Sigma)=d$, then there is an $\epsilon>0$ and a smoothly varying $d$-dimensional family of special Legendrians 
\[
\pi:\cM \rightarrow \{ x\in \mathbb{R}^d : |x|<\epsilon\}
\]
such that any fiber $\Sigma_x:= \pi^{-1}(x)$ has $\dim \mathfrak{H}_{m-1}(\Sigma_{x}) = 2m$, and \\$\dim \mathfrak{H}_{2m}(\Sigma_{x})= d$.
\end{itemize}
\end{defn}

Let us explain the relevance of this definition. A natural way to produce deformations of a special Lagrangian is through the symmetry group $SU(m)\ltimes \mathbb{C}^m$ acting on $\mathbb{C}^m$.  Let 
\[
\mu: \mathbb{C}^m \rightarrow \mathfrak{su}(m)\oplus \mathbb{C}^m
\]
be the moment map where we identify $\mathfrak{su}(m)\oplus \mathbb{C}^m \simeq \left(\mathfrak{su}(m)\oplus \mathbb{C}^m\right)^*$ using the standard inner product, which we denote simply by $\langle, \rangle$.  Then we have

\begin{lem}[Joyce \cite{JoyceII}, Lemma 3.4]
Let $v\in \mathfrak{su}(m)\oplus \mathbb{C}^m$, then $\langle \mu, v \rangle : \mathbb{C}^m \rightarrow \mathbb{R}$ is a harmonic function on any special Lagrangian submanifold in $\mathbb{C}^m$.
\end{lem}

In the case $\bfC = \Cone(\Sigma)$ it is straightforward to check that 
\begin{itemize}
\item The action of translation along a vector $v\in \mathbb{C}^m$ gives rise to a harmonic function of linear growth.
\item The action of rotation along a vector field $v \in \mathfrak{su}(m)$ gives rise to a harmonic function of quadratic growth.
\end{itemize}

By~\eqref{eq: sepVarHar} linear growth harmonic functions are in one-to-one correspondence with eigenfunctions on $\Sigma$ with eigenvalue $m-1$.  Since translation does not fix the cone $
\dim \mathfrak{H}_{m-1}(\Sigma) \geq 2m$.  Thus, condition $(i)$ in Definition~\ref{defn: strongInt} demands that each linear growth harmonic function on $\bfC$ is generated by a translation, and hence integrates to a deformation of $\bfC$.  The second condition $(ii)$ says that every infinitesimal special Legendrian deformation of $\Sigma$ integrates to an actual deformation, and each small deformation of $\Sigma$ is itself integrable.  Roughly speaking, $(ii)$ implies that $\Sigma$ is a smooth point in the moduli space of special Legendrian submanifolds of $S^{2m-1}\subset \mathbb{C}^m$ in a component of maximal dimension. 

In principle it seems difficult to determine when a given special Lagrangian cone is integrable in the sense of Definition~\ref{defn: strongInt}.  However, there is a {\em stronger} notion which can be checked in some examples.

\begin{defn}\label{defn: rigid}
We say that a special Lagrangian cone $\bfC \subset \mathbb{C}^m$ is {\bf rigid} if 
\begin{itemize}
\item[(i)] every harmonic function on $\bfC$ of linear growth is generated by an element of $\mathbb{C}^m$, and
\item[(ii)] every harmonic function on $\bfC$ of quadratic growth is generated by an element of $\mathfrak{su}(m)$.
\end{itemize}
\end{defn}

Clearly if $\bfC$ is rigid then it is integrable. We have the following lemma:

\begin{lem}\label{lem: rigid}
Suppose $\bfC\subset \mathbb{C}^m$, $m \geq 3$ is integrable in the sense of Definition~\ref{defn: strongInt}. Let $\mfC=\mathbb{R}^k\times \bfC$.  Then every harmonic function which is $W^{1,2}(\mfC\setminus \mfC_{sing})$ and has quadratic growth is either $(i)$ the pullback to $\mfC$ of a quadratic growth harmonic function on $\bfC$ or $(ii)$ generated by an element of $SU(m+k)$.   If $\bfC$ is rigid,  then every such harmonic function is generated by an element of $SU(k+m)$.
\end{lem}
\begin{proof}
Fix coordinates $(z_1,\ldots,z_k, z_{k+1},\ldots, z_{k+m})$.  Write $z_j = x_j+\sqrt{-1}y_j$.  We may assume that $y_{j}\big|_{\mfC}=0$ for $1\leq j \leq k$.  By Simon's real analyticity of Fourier Series \cite[Appendix 1]{SimonUn}, any harmonic function on $\mfC=\mathbb{R}^k\times \bfC$ which is $W^{1,2}(\mfC\setminus \mfC_{sing})$ and has quadratic growth is a linear combination of harmonic functions of the following type:
\begin{itemize}
\item quadratic growth harmonic functions on $\bfC$,
\item quadratic growth harmonic functions on $\mathbb{R}^k$,
\item linear combinations of functions of the form $x_{j}\cdot \phi $ where $1 \leq j \leq k$ and $\phi$ is a linear growth harmonic function on $\bfC$.
\end{itemize}
The first two cases are clear.  For the final case, integrability implies that any linear growth harmonic function on $\bfC$ can be written as a linear combination of the harmonic functions $x_i, y_i$, for $k+1 \leq i \leq k+m$.  Thus it suffices to show that the harmonic functions $x_jx_i, x_jy_i$ for $1 \leq j \leq k < i \leq k+m$ can be obtained from the action of $SU(k+m)$.  For this, note that if $ j \leq k < i \leq k+m$ 
\[
x_jx_i\big|_{\mfC} = (x_jx_i+y_jy_i) \big|_{\mfC}, \quad x_jy_i\big|_{\mfC} = (x_jy_i-y_jx_i) \big|_{\mfC}
\]
since $y_j\big|_{\mfC}=0$ for $j \leq k$.  However, the harmonic functions on the right hand side of each of these equations is easily seen to be generated by the action of $SU(k+m)$.
\end{proof}

\begin{rk}
The terms ``integrability" and ``rigidity" have appeared in several places in the literature on special Lagrangians, often with different meanings.  We warn the reader that the notions of integrability and stability introduced in Definitions~\ref{defn: strongInt} and~\ref{defn: rigid} differ from those in \cite{Haskins, JoyceII}.
\end{rk}

Suppose that $\bfC = \Cone(\Sigma)$ is integrable in the sense of Definition~\ref{defn: strongInt}, and let $v$ be a quadratic growth harmonic function on $\bfC$.  Since $\Sigma$ is a smooth point of the moduli space $\cM$ of special Legendrians we can pick a smooth Riemannian metric on $\mathcal{M}$ (eg. the the natural $L^2$ inner product).  Then, for any compact set $\mathcal{K} \subset \cM$ there is a $\delta >0$ such that, for any $\Sigma_{\kappa}, \kappa \in \mathcal{K}$ the exponential map
\[
\exp^{\cM}_{\Sigma_\kappa}: B_{\delta}(0) \rightarrow \mathcal{M}
\]
is well-defined, where $B_\delta(0)\subset \mfH_{2m}(\Sigma_\kappa) = T_{\Sigma_\kappa}\cM$ is the $\delta$-ball.

Let us now describe an example of a rigid special Lagrangian cone (c.f. Definition~\ref{defn: rigid})) discovered by Harvey-Lawson \cite{HL}.

\begin{ex}
The Harvey-Lawson $T^{m-1}$ cone is the special Lagrangian cone in $\mathbb{C}^m$ with an isolated conical singularity described by the equations
\[
\bfC^m_{HL}:= \{(z_1,\ldots,z_m) \in \mathbb{C}^m : |z_1|=\cdots =|z_m|, \quad {\rm Arg}(i^{m+1}z_1\cdots z_m)=0\}.
\]
\end{ex}

$\bfC^{m}_{HL}$ is a cone over a flat torus and so its spectrum can be explicitly computed; see \cite{Haskins, JoyceII, Marshall}.  We have the following result.

\begin{lem}\label{lem: specHLCone}
The spectrum of the $\bfC^{m}_{HL} = \Cone(T^{m-1})$ satisfies
\begin{itemize}
\item[$(i)$] The dimension of the linear growth harmonic functions on $\bfC^{m}_{HL}$ is given by
\[
\dim \mfH_{m-1}(T^{m-1}) = 2m
\]
\item[$(ii)$] The dimension of the space of quadratic growth harmonic functions on $\bfC^{m}_{HL}$ is given by
\[
 \dim \mfH_{2m}(T^{m-1}) = m^2-m  \qquad \text{ for } m \ne 8,9
 \]
 while, for $m=8$ we have $\dim \mfH_{16}(T^{7}) = 126$ and for $m=9$ we have $\dim \mfH_{18}(T^8) = 240$.
 \end{itemize}
 \end{lem}
 
\begin{proof}
The formula for $\dim \mfH_{2m}$ can be found in \cite[Table 6.1]{Marshall} for $m \leq 13$ and for general $m$ \cite[Section 3.2]{JoyceII}.  We were not able to find a reference for the result for $\dim \mfH_{m-1}$ in general (though the result for $m\leq 13$ is contained in \cite[Table 6.1]{Marshall}).  We only sketch the proof.  By the calculations in \cite{Haskins, JoyceII, Marshall} the eigenfunctions with eigenvalue $m-1$ are in one-to-one correspondence with points $(k_1,\ldots, k_{m-1}) \in \mathbb{Z}^{m-1}$ satisfying
\begin{equation}\label{eq: HLconeEV}
m\sum_{i=1}^{m-1}k_i^2 - \sum_{i,j=1}^{m-1} k_i k_j = m-1.
\end{equation}
The quadratic form $q(x) = m\sum_{i=1}^{m-1}x_i^2 - \sum_{i,j=1}^{m-1} x_ix_j$ on $\mathbb{R}^{m-1}$ has eigenvalue $1$ with multiplicity $1$, and eigenvector $(1,1,\cdots, 1)$, and eigenvalue $m$ with multiplicity $m-2$, with eigenvectors in the orthogonal complement of $(1,1,\ldots,1)$.  Let $e_i$ denote the standard basis of $\mathbb{R}^{m-1}$.  Then $\pm e_i, \pm (1,1, \ldots, 1)$ give $2m$ solutions of~\eqref{eq: HLconeEV}.  Thus it suffices to show these are the only solutions.  Given any $(k_1,\ldots, k_{m-1})$ we can write
\[
(k_1,\ldots,k_{m-1}) = \lambda (1,1,\ldots, 1) + v^{\perp}
\]
where $\lambda = \frac{1}{m-1} \sum_{i=1}^{m-1} k_i$ and $v^{\perp}$ is orthogonal to $(1,1,\ldots, 1)$.  From the eigenvalues of the quadratic form we deduce that
\[
m\sum_{i=1}^{m-1}k_i^2 - \sum_{i,j=1}^{m-1} k_i k_j=(m-1) = \lambda^2(m-1) + m |v^{\perp}|^2.
\]
Thus we see that $|\lambda| \leq 1$ and $|v^{\perp}|^2 <1$.  Furthermore, since $v^{\perp} = (k_1-\lambda, \ldots, k_{m-1}-\lambda)$ we deduce that each $k_i \in \{-1, 0, 1\}$, and each $k_i$ must be either $0$ or have the same sign as $\lambda$.  It therefore suffices to consider vectors of the form
\[
\mathbf{k} = (k_1,\ldots,k_{m-1})= \pm(1_p, 0_{m-1-p}) \quad \text{ where } \quad a_p = \overbrace{(a,\ldots, a)}^\text{$p$-times}.
\]
Without loss of generality we consider the $+$ case.  For such a vector we have $q(\mathbf{k}) = mp-p^2 = p(m-p)$. The only solutions to $p(m-p)= m-1$ for $1 \leq p \leq m-1$ are given by $p=1,m-1$.  After accounting for the obvious symmetries, the corresponding solutions are $e_i$ for  $1\leq i \leq m-1$, and $(1,\ldots, 1)$.  Including the case of $-(1_p, 0_{m-1-p})$ yields the desired conclusion.
\end{proof}

Finally, we have
\begin{cor}
For $m\ne 8, 9$, the cone $\bfC^{m}_{HL}$ is rigid in the sense of Definition~\ref{defn: rigid}.
\end{cor}
\begin{proof}
By Lemma~\ref{lem: specHLCone} we have already verified property $(i)$ of Definition~\ref{defn: rigid}.  To verify property $(ii)$ it suffices to observe that the subgroup of $SU(m)$ preserving $\bfC^m_{HL}$ is $U(1)^{m-1}$, and hence the dimension of the space of quadratic growth harmonic functions on $\bfC^{m}_{HL}$ induced by $SU(m)$ is given by
\[
\dim SU(m) -(m-1) = m^2-1-(m-1) = m^2-m
\]
and hence by Lemma~\ref{lem: specHLCone}, if $m\ne 8,9$, $\bfC^{m}_{HL}$ is rigid.
\end{proof}

We end by remarking that it seems to be unknown whether the Harvey-Lawson cones $\bfC^{8}_{HL}, \bfC^{9}_{HL}$ are integrable in the sense on Definition~\ref{defn: strongInt}.  When $m=8$ there is a $126-8\cdot 7 = 70$ dimensional space of excess quadratic growth harmonic functions, while when $m=9$ there is a $240-9\cdot 8 = 168$ dimensional space.  Due to the rather large number of excess quadratic growth harmonic functions it would seem surprising if every infinitesimal deformation turned out to be integrable.

\subsection{Harmonic functions on exact special Lagrangians}

Suppose that $N$ is a closed integral current of multiplicity $1$ which is special Lagrangian in $\mathbb{C}^n$.  Since $N$ is area minimizing \cite{HL}, Almgren's big regularity theorem \cite{Almgren, DLSV} implies $N$ is smooth outside a set of Hausdorff dimension at most $n-2$.   We will denote by $N_{reg}$ the smooth part of ${\rm supp}(N)$.  By \cite[Lemma 33.2]{SimonBook} the varifold underlying $N$ is stationary.

\begin{defn}\label{defn: exact}
An {\bf exact special Lagrangian} is a multiplicity $1$, closed integral current $N$ which is special Lagrangian and such that $\frac{1}{2}\lambda|_{N_{reg}} = d\beta|_{N_{reg}}$ for some some function $\beta: N_{reg} \rightarrow \bR$. If ${\rm supp}(N)$ is connected, the function $\beta$ is unique up to addition of a constant. 
\end{defn}

\begin{rk}\label{rk: betaNorm}
In Section~\ref{sec: Decay} we will fix scale dependent normalizations for $\beta$.  For this reason, it is convenient to state the results of the first sections of this paper without reference to a particular choice of normalization.
\end{rk}

Recall that if $T$ is a $k$-varifold in $\mathbb{R}^n$, then a function $u$ on $\mathbb{R}^n$ is said to be weakly harmonic (resp. subharmonic) on $T$ if, for any smooth function $\eta$ with compact support we have
\[
\int_{\mathbb{R}^n\times G(k,n)} \langle \nabla^{\omega}\eta, \nabla^{\omega}u \rangle dT(x,\omega) =0 \qquad (\text{ resp. } \leq 0).
\]
It is well-known that, since the varifold underlying $N$ is stationary, the coordinate functions $x_i, y_i$ define weakly harmonic functions on $N$ in the sense of varifolds (see, e.g. \cite[Chapter 3]{CMbook}).  We will also need the following result

\begin{lem}
If $N$ is an exact special Lagrangian current then $\Delta_N \beta =0$ in the weak sense, and strongly on $N_{reg}$.
\end{lem}
\begin{proof}
The key point is that, on $N^{reg}$ we have $\nabla^{N}\beta = (J\ux)^{T}$, where $\ux$ denotes the position vector and $(J \ux )^{T}$ denotes the projection to the tangent space of $N$. Fixing a smooth, compactly supported function $\eta$ and applying stationarity to the vector field $\eta J\ux$ yields the result.  Alternatively, one can argue directly as in Lemma~\ref{lem: subharmonicFunctions} below.%
\end{proof}

\begin{lem}\label{lem: subharmonicFunctions}
Let $N$ be an exact special Lagrangian current.  Then the functions $x_i^2, y_i^2, \beta^2$ are weakly subharmonic on $N$, in the sense of varifolds.
\end{lem}
\begin{proof}
Note that this claim is obvious when $N$ is smooth, since $x_i, y_i, \beta$ are harmonic.  Furthermore, the claim regarding the $y_i^2$ can be easily obtained by applying the stationarity of $N$ to the (globally defined) vector field $y_i \frac{\del}{\del y_i}$ (and similarly for $x_i^2$).  Thus we will only prove the statement for $\beta^2$.  The main difficulty is that $\beta$ is only defined on $N$, and hence $\beta^2$ does not obviously have a well-defined gradient vector field in a neighborhood of $N$.  Since the problem is local, we may use the fact that, for any $R >0$, the functions  $\beta, |\nabla \beta|$ are uniformly bounded in $B_{R}\cap N_{reg}$, and satisfy $\Delta_{N} \beta^2 \geq 0$.  Clearly it suffices to prove the result near $N_{sing}:= N\setminus N_{reg}$.  Fix any $\epsilon >0$ and let $\phi \geq 0$ be a smooth function with compact support in $B_{R/2}$, supported near $N_{sing}\cap B_{R/2}$. Fix $\alpha>0$ to be determined and fix $0<\epsilon \ll \delta \ll1$.  Since $N_{sing}$ has Hausdorff dimension at most $n-2$ we can cover $N_{sing}\cap B_{R/2}$ by countably many balls $B_{r_i}(p_i)$ (with points $p_i \in B_{R/2}$) such that
\[
\sum_{i}r_i^{n-2+\alpha} <2^{-\alpha}\epsilon.
\]
Let $\eta_i:\mathbb{C}^n \rightarrow [0,1]$ be a smooth function such that
\[
\eta_i = \begin{cases} 0 &\text{ in } B_{r_i}(p_i)\\
1 & \text{ in }B_{2r_i}^{c}(p_i)
\end{cases}
\]
and such that $|\nabla \eta_i| \leq 2r_i^{-1}$.  Define $\eta = \prod_i \eta_i$ and note that this product is well-defined and smooth.  Then we have
\[
-\int_{N}  \eta \langle \nabla^{N}\phi , \nabla^{N}\beta^2 \rangle = \int_{N} \eta\phi \Delta_{N}\beta^2  + \int_{N}\phi\langle \nabla^{N}\eta,\nabla^{N}\beta^2 \rangle
\]
Now since $|\beta| + |\nabla^N\beta|\leq C$ on $N\cap B_{R}$ we have
\[
\big|\int_{N}\phi\langle \nabla^{N}\eta,\nabla^{N}\beta^2\rangle \big| \leq C\sum_{i} r_{i}^{-1} \cH^n(N\cap B_{2r_i}(p_i))
\]
Now since $p_i \in B_{R/2}$ and $\epsilon \ll \delta$ we can arrange that $2r_i\ll R$ and so $B_{2r_i}(p_i) \subset B_{R}$.  Thus, by volume monotonicity we have
\[
\cH^n(N\cap B_{r_i}(p_i))\leq C\cH^n(N\cap B_{R})r_i^{n}
\]
 for a constant $C$ independent of $\epsilon$.  Since $\Delta_{N}\beta^2 \geq 0$ on $N_{reg}$ we obtain
\[
-\int_{N}  \eta \langle \nabla^{N}\phi , \nabla^{N}\beta^2\rangle  \geq  -C\sum_{r}r_i^{n-1} \geq -C\epsilon
\]
provided we take $\alpha <1$.  Taking $\epsilon \rightarrow 0$ yields the result.
\end{proof}

\subsection{Darboux coordinate systems and local potentials}\label{sec: Darboux}

The following discussion is standard; see, for example, \cite{JoyceV}. It follows from Weinstein's tubular neighborhood theorem that if $N, N' \subset (X, \omega)$ are $C^1$ close Lagrangian submanifolds then $N'$ can be identified with the graph of a closed $1$-form $\alpha:N\rightarrow T^*N$.  Since closed $1$-forms are locally exact we can view $C^1$ close Lagrangians as locally corresponding to smooth functions on $N$.  When $N, N'$ are special Lagrangian, we have the following consequence of \cite[Proposition 2.13]{JoyceV}

\begin{lem}\label{lem: JoyceLemma}
If $N$ is special Lagrangian, and $N'$ is a $C^1$-close special Lagrangian, then locally $N'= {\rm graph}(df)$ where $f$ is a locally defined function on $N$ satisfying
\[
\Delta f = Q(x,df, \nabla^2f)
\]
where $Q(x,y,z) = O(|y|^2+|z|^2)$ for small $y,z$.
\end{lem}

We now recall precisely the sense in which special Lagrangians which are close in an appropriate topology to the special Lagrangian cone $\mathbf{C}$ can be described as graphs of local potential functions.  

\subsubsection{The case of a cone with smooth link}\label{sec: DarbouxSmoothLink}

Suppose $\bfC = \Cone(\Sigma) \subset \bC^m$ is a special Lagrangian cone with an isolated singularity at $0 \in \bC^m$.  The link $\Sigma:=\bfC \cap S^{2m-1}$ is a Legendrian submanifold of $S^{2m-1}$ with its standard Sasaki-Einstein structure.   Let
\[
r^2= \sum_{i=1}^{m} |z_i|^2
\]
be the standard radial function.  The symplectic form on $\bC^m$ can be written as
\[
\omega = \frac{1}{2}\ddb r^2 =-\frac{1}{4}dJdr^2 = d(r^2\eta) =2rdr\wedge \eta + r^2d\eta 
\]
where
\[
\eta= \frac{1}{2r^2}\sum_{i=1}^{m}x_idy_i-y_idx_i= \frac{1}{2r^2}\lambda
\]
is the pull-back of the standard contact form on $S^{2m-1}$.  Since $\bfC$ is a Lagrangian cone, the Euler vector field $r\frac{\del}{\del r}$ is tangent to $\bfC$.  Thus, from $\omega|_{\bfC}=0$ we conclude that
\[
\eta|_{\Sigma}=0, \qquad d\eta|_{\Sigma}=0.
\]
In particular $\Sigma$ is transverse to the fibers of the Hopf fibration.  Furthermore, $d\eta$ is precisely the pull-back to $S^{2m-1}$ along the Hopf fibration of the Fubini-Study metric $\omega_{FS}$ on $\mathbb{P}^{m-1}$.  Thus, at least locally, we can view $\Sigma$ as a lift to $S^{2m-1}$ of a Lagrangian submanifold of $\mathbb{P}^{m-1}$.  Let $\Sigma'=\Sigma/S^1$ be the projection to a Lagrangian submanifold of $(\mathbb{P}^{m-1},\omega_{FS})$.  Fix a point $x \in \Sigma$ and let $x'\in \Sigma'$ be the image of $x$ under the projection map.  Let $\{p_1,\ldots, p_{m-1}, q_1,\ldots, q_{m-1}\}$ be local Darboux coordinates for $(\mathbb{P}^{m-1}, \omega_{FS})$ centered on $x' \in \Sigma'$ and such that
\[
\omega_{FS}= \sum_{i=1}^{m-1} dp_i\wedge dq_i ,\qquad \Sigma' \overset{locally}{=} \{p_1=\cdots=p_{m-1}=0\}.
\]
Fixing a local branch of the covering $\Sigma \rightarrow \Sigma'$ and pulling back along the Hopf fibration we have $d(\eta-\sum_ip_idq_i) = 0$, while
\[
(\eta-\sum_ip_idq_i) (J\frac{\del}{\del r})=1.
\]
Thus we can find a local coordinate $p_m$ on $S^{2m-1}$ so that $\eta-\sum_ip_idq_i=dp_m$ and  $(p_1,\ldots,p_m,q_1,\ldots, q_{m-1})$ forms a local coordinate system on $S^{2m-1}$.  Furthermore, since $\eta\big|_{\Sigma}=0$ we can assume that $p_m=0$ on $\Sigma\subset S^{2m-1}$.  Define local Darboux coordinates on $\mathbb{C}^m$ by
\[
\begin{aligned}
p_i' &= r^2p_i,\quad  q_i'=q_i,\quad \text{ for } 1 \leq i \leq m-1\\
 p_m'&=-p_m, \quad q_m'=r^2.
 \end{aligned}
\]
Then we have
\begin{equation}\label{eq: contactDarboux}
r^2\eta = \sum_{i=1}^{m-1}p_i'dq_i'-q_m'dp_m' 
\end{equation}
and hence $(p_i',q_i')_{\{1\leq i \leq m\}}$ form a system of Darboux coordinates on the set of points $U'= \Cone(U)$.  Furthermore, on this set we have
\[
\bfC\cap U' = \{ p_1'=\cdots = p_m'=0\}.
\]
We will call the Darboux coordinate systems constructed in this way {\bf adapted Darboux coordinates}.

\begin{rk}\label{rk: DarbouxinFam}
It is an easy consequence of the implicit function theorem that the Darboux coordinates can be constructed in such a way that they vary smoothly with respect to smoothly varying families of special Legendrians.
\end{rk}

Fix, once and for all a finite cover of $\bfC\cap S^{2m-1}$ by contractible open sets $U_{\alpha}\subset S^{2m-1}$ such that $V_{\alpha}:= \Cone(U_\alpha)$ admits an adapted Darboux coordinate system as constructed above and such that $\frac{1}{4}U_{\alpha}$ still cover $\bfC\cap S^{2m-1}$.  Relative to this choice we make the following definition.

\begin{defn}\label{defn: graph}
We say that a Lagrangian $N$ is the {\bf graph associated to a local potential function} over an open set $U\subset \bfC$ if, on the open set $V_{\alpha} \cap  U$, $N$ is given in the Darboux coordinate system as the graph of a $1$-form $df_{\alpha}$ for a locally defined function $f_{\alpha}$. 
\end{defn}

Suppose that $N$ is a special Lagrangian in $\bC^m$ which is given locally in the Darboux coordinates $(p',q')$ as a graph $\{p_i'=\frac{\del f}{\del q_i'}\}$. Let $x' \in \bfC\cap S^{2m-1}$ be a fixed point and let $B$ denote a suitably chosen ball in $\mathbb{C}^m$ containing $x'$ so that the coordinates $(p',q')$ yield a symplectomorphism
\[
\phi: (B, \omega) \rightarrow (\phi(B), \omega_{std})
\]
onto $\phi(B)$, which is open in $\mathbb{C}^m$ and suppose that $\phi(q',0) = x'$.  Since $\bfC=\{p'=0\}$ it is straightforward to show that
\[
\frac{d}{dt}\bigg|_{t=0} \phi(q', t \frac{\del f}{\del q '}) = J{\rm grad} f.
\]
Using the additive structure on $\mathbb{C}^m$ we see that
\[
N \ni \phi(q',  \frac{\del f}{\del q '}) = \phi(q',0) + J{\rm grad} f(q') + E,
\]
where $E$ is some vector of length $|E| \leq C|df|^2$.  We remark that in our applications $|df|$ will be small. This observation easily yields

\begin{lem}\label{lem: distEstPot}
Suppose $N$ is a Lagrangian which is the graph associated to a potential function over an open set $U \subset \bfC$ and assume that in any local coordinate chart $V_{\alpha}$ the potential function satisfies
\[
r^{-1}|df_{\alpha}|+|D^2f_{\alpha}| \leq \epsilon
\]
 Given two Darboux coordinate systems $(\phi^1, V_1)$ and $(\phi^2, V_2)$ with $\phi^1(V_1)\cap\phi^2( V_2) \subset U$, let $f_i$ be the corresponding local potential functions of $N$.  Then there are constants $C, \epsilon_0(\bfC)>0$, depending only on the coordinates $\phi^i$  such that if $\epsilon < \epsilon_0$ then
\[
r^{-1}|d(f_1-f_2)| \leq C\epsilon^2
\]
\end{lem}

Recall that the functions $f_\alpha$ are only defined up to a constant.  In order to fix this ambiguity we invoke the exactness of $N$.  Recall that in the adapted Darboux coordinates, the Liouville form is given by ~\eqref{eq: contactDarboux}. Since $N$ is exact, there is a function $\beta:N \rightarrow \mathbb{R}$ such that
\[
\begin{aligned}
d\beta = r^2\eta\big|_{N} &=  \sum_{i=1}^{m-1}\frac{\del f_{\alpha}}{\del q_i'}dq_i'-q_m'd\left(\frac{\del f_{\alpha}}{\del q_m'}\right)\\
&=d\left(f-q_m'\frac{\del f}{\del q_m'}\right).
\end{aligned}
\]
Thus, there is a constant $c$ depending on the coordinates so that 
\[
\begin{aligned}
\beta(q') &= c+f(q')-q_m'\frac{\del f}{\del q_m'} \\
&= c-\frac{1}{2}r^3\frac{\del}{\del r} \left(\frac{f}{r^2}\right)
\end{aligned}
\]
By adjusting $f$ by a constant we can assume $c=0$.

\begin{defn}\label{defn: normalized}
We say that the exact Lagrangian $N$ is a graph associated to a {\bf normalized local potential function} over $U$ if $N$ is the graph of a local potential function and the potential functions $f$ are normalized to satisfy
\[
\beta(q')=-\frac{1}{2}r^3\frac{\del}{\del r} \left(\frac{f}{r^2}\right).
\]
Throughout this paper all potential functions will be taken to be normalized in this sense.
\end{defn}

\begin{rk}\label{rk: changeOfNorm}
Note that a choice of normalization for $\beta$ fixes a choice of normalization for the local potentials $f$.  Furthermore, if we change $\beta \mapsto \beta +c$, then the local potentials are changed by $f \mapsto f+c$.
\end{rk}

\subsubsection{The case of a cylinder}

We now extend this discussion to the case of a cylinder $\mfC:=\mathbb{R}^k\times \bfC\subset \mathbb{C}^n$.  We proceed as in Section~\ref{sec: DarbouxSmoothLink} in the $\bC^{n-k}$ factor to find open sets $U_{\alpha}\subset S^{2(n-k)-1}$ covering $\Sigma = \bfC \cap S^{2(n-k)-1}$ and such that $V_{\alpha}=\Cone(U_\alpha)$ admits an adapted Darboux coordinate system and $\frac{1}{4}U^{\alpha}$ still cover $\bfC\cap  S^{2(n-k)-1}$.  For any $U_{\alpha}$ let $(p_i',q_i')_{\{1\leq i \leq n-k\}}$ denote the Darboux coordinates.  In the $\bC^k$ factor we take the standard coordinates $(x_i,y_i)$ so that $\omega_{std}= \sum_{i=1}^{k}dx_i\wedge dy_i$, and $\bR^k\subset \bC^k = \{y_1=\cdots=y_k=0\}$.  These choices yield Darboux coordinates, still denoted $\phi^{\alpha}$, defined on $V_\alpha' = \mathbb{C}^k\times V_{\alpha}$.   The standard Liouville form is given by
\[
\frac{1}{2}\lambda = r^2\eta + \frac{1}{2}\sum_{j=1}^{k}(x_jdy_j-y_jdx_j).
\]
If $N$ is an exact Lagrangian given locally as a graph over $\mfC$ by
\begin{equation}\label{eq: GraphCylinder}
N  = \left\{ p_i'=\frac{\del f}{\del q_i'} , \quad y_j= -\frac{\del f}{\del x_j} \right\}
\end{equation}
in our local Darboux coordinates, then we have
\[
\begin{aligned}
d\beta &= \sum_{i=1}^{n-1}\frac{\del f}{\del q_i'}dq_i'-q_n'd\left(\frac{\del f}{\del q_n'}\right) - \frac{1}{2}\sum_{j=1}^{k}\left(x_j d\left(\frac{\del f}{\del x_j}\right) - \frac{\del f}{\del x_j}dx_j\right)\\
&=d\left(f-q_n'\frac{\del f}{\del q_n'}-\frac{1}{2}\sum_{j=1}^k x_j\frac{\del f}{\del x_j}\right).
\end{aligned}
\]
As in section~\ref{sec: DarbouxSmoothLink} above we have
\[
q_n'\frac{\del f}{\del q_n'}= \frac{1}{2}r\frac{\del f}{\del r}
\]
where $r\frac{\del}{\del r}$ is the radial vector field on $\bC^{n-k}$.  In particular, if we set $R$ to be the radial function on $\bC^{n}$ then we see that, since $\mfC\subset\{y_j=0: 1 \leq j \leq k\}$ we have
\[
q_n'\frac{\del f}{\del q_n'}+\frac{1}{2}\sum_{j=1}^k x_j\frac{\del f}{\del x_j} = \frac{1}{2}R\frac{\del f}{\del R}.
\]
It follows that
\begin{equation}\label{eq: betaVsPot}
\begin{aligned}
\beta=c-\frac{1}{2}R^3\frac{\del}{\del R}\left(\frac{f}{R^2}\right)
\end{aligned}
\end{equation}
where $c$ is a constant.  The following lemma shows that if we choose the potentials to be normalized, then the error on the overlaps is quadratically suppressed.

\begin{lem}\label{lem: diffPotEst}
There is an $\epsilon_0$ small, depending only on $\bfC$, with the following effect.  Suppose $N$ is an exact Lagrangian which is the graph associated to a local potential function over an open set $U \subset (\mathbb{R}^k \times \bfC) \cap B_2$, in the sense of Definition~\ref{defn: graph}.  Suppose moreover that the local potential functions $f_{\alpha}$  satisfy
\[
r^{-1}|df_{\alpha}|+|D^2f_\alpha| \leq \epsilon
\]
for $\epsilon < \epsilon_0$.  Let $(\phi^{i}, V'_{i}), i=1,2$ be two sets of local Darboux coordinates as constructed above and let $f_i, i=1,2$ be the local normalized potential functions. Then, there is a constant $C$, depending only on $\bfC$, such that
\[
\sup_{\phi^{1}(V'_{1})\cap \phi^{2}(V'_{2})\cap U} r^{-2}|f_{1} -f_{2}| \leq  C\epsilon^2.
\]
\end{lem}
\begin{proof}
Fix a point $(z_*,w_*) \in \phi^{1}(V'_{1})\cap \phi^{2}(V'_{2})\cap U$.  By rescaling and translating we may assume that $|z_*|=1$ and $w_*=0$.  Fix points $q_i \in V'_i$ so that
\[
(z_*,w_*) = \phi_i(q_i, 0).
\]
It follows from the obvious generalization of the estimate in Lemma~\ref{lem: distEstPot} to the case of cylinders that the points $\phi^1(q_1, df_1), \phi^2(q_2, df_2) \in N$ may be joined by a path in $N$ of length $O(\epsilon^2)$.   Now, recall that $d\beta =\frac{1}{2}\lambda$ for the Liouville $1$-form $\lambda$.  Since $\lambda$ is uniformly bounded in $B_2\subset \mathbb{C}^n$ we have
\[
\begin{aligned}
|\beta(\phi^1(q_1, df_1))- \beta(\phi^2(q_2, df_2))| &\leq C  \epsilon^2.
\end{aligned}
\]
Since the potentials $f_i$ are normalized, it follows from~\eqref{eq: betaVsPot} that
\[
\big|f_1(z_*,w_*)-f_2(z_*,w_*)\big| \leq |\beta(\phi^1(q_1, df_1))- \beta(\phi^2(q_2, df_2))| + \frac{1}{2}{R}\big|\frac{\del}{\del R} (f_1-f_2)\big|
\]
Now $R=1$ by assumption, and by Lemma~\ref{lem: distEstPot} we have
\[
|\frac{\del}{\del R} (f_1-f_2)\big| \leq C\epsilon^2.
\]
Combining these estimates yields the result.
\end{proof}

\begin{rk}
Following Remark~\ref{rk: DarbouxinFam} it is not hard to see that the constants $C,\epsilon_0$ can be chosen uniform for compact sets $\mathcal{K}\subset \mathcal{M}$. 
\end{rk}

\begin{rk}\label{rk: potLocalDefined}
Throughout this paper, $f:= (f_{\alpha})$ will denote a collection of potential functions defined in the fixed Darboux coordinate charts $(\phi^\alpha, V'_\alpha)$.  Any norm of $f$ is then supremum of the norms of $f_{\alpha}$ in the associated charts; for example
\[
\|f\|_{L^{2}(U)} := \sup_{\alpha} \|f^\alpha\|_{L^{2}(U\cap V'_{\alpha})}.
\]
 For most of our argument it will not be essential that the $f_{\alpha}$ are locally defined with the exception of Proposition~\ref{prop: propOfSmall} where we will use that the Darboux coordinates are defined in $\mathbb{C}^{k}\times \Cone(U_{\alpha})$.
\end{rk}

\section{Existence of Local Potentials}\label{sec: pot}

We now establish the existence of local potentials in the sense of the previous section, together with some estimates for these potentials.  The following lemma concerning multiplicity one convergence is an immediate consequence of Allard's regularity theorem \cite{Allard, SimonBook}, together with the fact that weak convergence of area minimizing, integral currents implies varifold convergence \cite[Theorem 34.5]{SimonBook}.

\begin{lem}\label{prop: existGraph}
Suppose  $N_i$ is a sequence of special Lagrangians such that $\lim_{i\rightarrow \infty}N_i = \mathfrak{C}$ as closed integral currents. For all  $\eta, \tau \in (0,1]$, $\rho>0$ there exists $k = k(\eta, \tau)$ such that, for all $i \geq k$, the following holds:  $N_i\cap B_{2\rho}\cap\{r>2\rho\tau\}$ is the graph associated to a potential function $f$ defined on an open neighborhood of $ \mfC \cap B_{2\rho}\cap\{r>2\rho\tau\}$ satisfying the estimate 
\[
\sup_{\mfC\cap B_{2\rho}\cap\{r>2\rho\tau\}}r^{-1}|df| + |D^2f| \leq \eta.
\]
\end{lem}

The next result is a quantitative improvement of Lemma~\ref{prop: existGraph} which says that if a special Lagrangian $N$ is sufficiently close to $\mathfrak{C}$ on an annulus then in fact $N$ is a graph over $\mfC$ on an extended annulus.

\begin{prop}\label{prop: existPot1}
Fix $\rho>0$ and $\eta, \tau, \mu \in (0, \frac{1}{10})$ and  $\gamma \in(\frac{1}{2},1)$, and let $\mathcal{K}\subset \mathcal{M}$ be a compact set in the moduli space of special Legendrians.  Let $\mfC= \mathbb{R}^k \times \Cone(\Sigma_{\kappa})$ for some $\kappa \in \cK$.  There exists a number $\eta_1 := \eta_1(\cK,\eta, \tau, \mu, \gamma) $ with the following property: if $N\subset \mathbb{C}^n$ is a special Lagrangian such that 
\begin{equation}\label{eq: volBndProp}
\mathcal{H}^n(N\cap B_{2\rho}) \leq \mathcal{H}^n(\mathfrak{C}\cap B_{2\rho}) + \omega_n(1-\gamma)^n\rho^n
\end{equation}
and $N \cap B_{2\rho} \cap \{r>2\rho\tau\}$ is the graph associated to a potential function $f$ defined in an open neighborhood of $\mfC \cap B_{2\rho}\cap \{r> 2\rho\tau\}$ and satisfying the estimates
\[
\sup_{\mfC \cap B_{2\rho}\cap \{r> 2\rho\tau\}} r^{-1}|df| + \sup|D^2f| < \eta_1
\]
then $N\cap B_{2\gamma \rho}\cap  \{r>2\gamma \rho(\mu\tau)$ is associated to a potential function $F$ defined on an open neighborhood of $\mfC \cap B_{2\gamma\rho} \{r>2\gamma\rho(\mu \tau)\}$, satisfying the estimate
\[
\sup_{\mfC\cap B_{2\gamma \rho}\cap  \{r>2\gamma \rho(\mu\tau)\}} r^{-1}|dF| + |D^2F| < \eta.
\]
\end{prop}

\begin{proof}
The proof is by contradiction. By rescaling we may assume that $\rho=1$.  Fix $\mu \in (0,\frac{1}{10})$. Suppose we cannot find $\eta_1$ small so that the graphical extension $F$ exists and satisfies the desired estimate. Then there is a sequence of special Lagrangians $N_i \subset \mathbb{C}^n$, and cylinders $\mfC_i = \mathbb{R}^k\times \Cone(\Sigma_i)$ for $\Sigma_i\in \mathcal{K}$ satisfying~\eqref{eq: volBndProp} and having the property that $N_i \cap B_2\cap \{r>2\tau\} $ is the graph associated to a potential function $f_i$ satisfying the estimates
\[
\sup|df_i| + \sup|D^2f_i| \leq \frac{1}{i}.
\]
but $N_i \cap B_{2\gamma}\cap \{r>2\gamma\mu\tau\}$ is not the graph of a potential function satisfying the desired estimate.  Since $\Sigma_i$ vary in a compact family there is a special Legendrian $\Sigma_{\infty} \in \cK$ such that $\mfC_i \rightarrow \mfC = \mathbb{R}^k\times \Cone(\Sigma_{\infty})$ smoothly away from $\mathbb{R}^k\times \{0\}$.  By the volume condition~\eqref{eq: volBndProp}, we can take a limit in the sense of closed integral currents \cite[Theorem 27.3]{SimonBook}, $N_{i}\rightarrow N_{\infty}$.  From the assumption on the potential functions $f_i$ we see that, in compact subsets of $B_2\cap \{r>2\tau\} $, $N_i$ converges smoothly to $\mfC$ with multiplicity $1$.  Let $N_{\infty}^{{\rm reg}}$ denote the regular set of ${\rm supp}(N_{\infty})$, which is dense by \cite[Theorem 36.2]{SimonBook}.  Since $\del N_{\infty} =0$ in $B_2$, Almgren's regularity theorem gives that $N_{\infty}^{\rm sing} = {\rm supp}(N_{\infty})\setminus N_{\infty}^{{\rm reg}}$ has Hausdorff codimension at least $2$ in ${\rm supp}(N_{\infty})$.  

Denote by $N_{\infty}^{{\rm reg}, 0}$ the path connected component of $N_{\infty}^{{\rm reg}}$ containing \\$N^{{\rm reg}}_{\infty} \cap \mfC \cap B_{2}\cap \{r>2\tau\}$.  We claim that $N_{\infty}^{{\rm reg}, 0} \subset \mfC$.  Indeed, by the regularity theory for the minimal surface system, both $\mfC$ and $N_{\infty}^{{\rm reg}, 0}$ are real analytic subsets of $\mathbb{C}^n$.  Since $N_{\infty}^{{\rm reg},0} $ agrees with $\mfC$ on an open subset, the result follows immediately from the standard fact that a real analytic function vanishing on an open set is identically zero.

Next we claim $\mfC \subset \overline{N_{\infty}^{reg,0}}$.   Choose any point $p_1$ in $\mfC\cap \{r\leq 2\tau\} \cap B_{2}$, and let $p_0 \in N_{\infty}^{reg, 0} \cap \mfC\cap B_{2}\cap \{r>2\tau\}$.  Since $N_{\infty}^{sing}$ has Hausdorff dimension at most $n-2$, we can find a smooth curve $\gamma(t), t\in[0,1]$ from $p_0$ to $p_1$, lying in $\mfC$, and such that $\gamma(t)$ avoids $N_{\infty}^{sing}$ for $t\in[0,1)$.  Clearly $\gamma(t) \in \mfC\cap N_{\infty}^{reg,0}$ for $t\in [0, \epsilon)$.  Let $T^*>0$ to be the first time such that $\gamma(T^*)\notin N_{\infty}^{reg, 0}$.  If $T^*<1$, then since $N_{\infty}$ is relatively closed we have $\gamma(T^*)\in N_{\infty}$.  On the other hand, $\gamma(T^*)\notin N_{\infty}^{sing}$ by assumption.  Hence $\gamma(T^*)\in N_{\infty}^{reg}$, and therefore is necessarily a point in $N_{\infty}^{reg, 0}$, a contradiction.  Therefore, $\gamma(t) \in N_{\infty}^{reg,0}$ for all $t<1$, and the claim follows. Now since $N_{\infty}^{reg,0}$ is equal to $\mfC$ with multiplicity $1$ in a large open set, it follows that $N_{\infty}^{reg,0}$ has multiplicity one everywhere.

Now suppose there exists a point $p \in B_{2\gamma} \cap {\rm supp}(N_{\infty}) \setminus \mfC$.  From the preceding discussion we must have $p \in B_{2\gamma} \cap \{r\leq 2\tau\}$, and we may as well assume $p$ is a regular point of $N_{\infty}$ (since such points are dense).   For some $\delta >0$, $B_{1-\gamma+\delta}(p) \subset B_2$, and so the monotonicity formula yields
\[
\mathcal{H}^n(N_{\infty}\cap B_2) \geq \mathcal{H}^n(\mfC\cap B_2) + \omega_n(1-\gamma+\delta)^n,
\]
contradicting the volume assumption.  Thus, $N_\infty = [\mfC]$ with multiplicity $1$ in $B_{2\gamma}$. Now the result follows from Lemma~\ref{prop: existGraph}. 
\end{proof}

We next state a regularity result for special Lagrangian graphs over $\mfC$.

\begin{lem}\label{lem: ellEst}
Fix $\gamma \in (\frac{1}{2}, 1)$ and let $\mathcal{K}\subset \cM$ be a compact set of smoothly varying special Legendrians.  There exist constants $\eta_2(\cK), C(\cK, \gamma)$ with the following effect: if $\mfC= \mathbb{R}^k\times \Cone(\Sigma)$ for some $\Sigma \in \cK$ and $N$ is a special Lagrangian which is the graph associated to a potential function $f$ over $ \mfC\cap B_2 \cap \{r > 2\tau\}$ such that $f$ satisfies the estimates
\[
\sup_{\mfC\cap B_2 \cap \{r > 2\tau\}}r^{-1}|df| + |D^2f| \leq \eta_2,
\]
then we have the estimates
\begin{equation}\label{eq: perturbEst}
\sup_{B_{2\gamma}\cap \{r >4\tau\}} r^{-1}|df| + |D^2f| \leq C\sup_{B_2 \cap \{r >2\tau\}}r^{-2}|f|
\end{equation}
and
\begin{equation}\label{eq: LinftyL2PotEst}
\sup_{B_{2\gamma} \cap \{r> 4\tau\} }|f| \leq C\tau^{-n/2}\|f\|_{L^{2}(B_{2}\cap \{r> 2\tau\})}
\end{equation}
\end{lem}
\begin{proof}
The first claim is a straightforward consequence of the Cordes-Nirenberg estimate \cite{CaffarelliCabre} and the Schauder theory \cite{GT}.  The second claim follows from the local maximum principle for solutions of elliptic PDEs \cite[Theorem 9.20]{GT} together with scaling and translating.

\end{proof}

At this point we have the necessary ingredients to prove the main quantitative extension result for the potential.  We introduce some notation to make the exposition more efficient.  Suppose $N$ is an exact special Lagrangian which is associated to a potential function over $\mfC \cap B_{\rho}\cap \{r>\tau \rho\}$.  We define the following quantities, where the underline signifies scale invariance:
\begin{equation}\label{eq: normDefn}
\begin{aligned}
Av_{N}(\beta, \rho) &= \frac{1}{{\rm Vol}(N\cap B_{\rho})} \int_{N\cap B_{\rho}}\beta\\
\|\underline{\beta}\|^2_{L^{2}(N\cap B_\rho)}  &= \rho^{-n-4}\int_{N\cap B_\rho}\beta^2\\
\|\underline{y}\|^2_{L^{2}(N\cap B_{\rho})} &= \rho^{-n-2}\sum_{i=1}^{k}  \int_{N\cap B_{\rho}} y_{i}^2\\
\|\underline{\beta}, \underline{y}\|_{L^{2}(N\cap B_{\rho})} &= \|\underline{\beta}\|_{L^{2}(N\cap B_\rho)} +\|\underline{y}\|_{L^{2}(N\cap B_{\rho})}\\
\|\underline{f}\|^2_{L^{2}}(\rho, \tau) &= \rho^{-n-4}\int_{\mfC\cap B_{\rho} \cap \{r>\tau \rho\}} f^2\\
\end{aligned}
\end{equation}

\begin{defn}
Consider a special Lagrangian cylinder of the form $\mfC:= \mathbb{R}^k\times \bfC$.  Fix $\gamma \in [\frac{1}{10}, 1)$.  We define the following properties that an exact special Lagrangian $N$ may possess:
\begin{itemize}
\item[(i)] We say that $N$ has the {\bf small graph property} $P_1(\eta, \tau, \delta)$ with respect to $\mfC$ if $N$ is the graph associated to a normalized potential $f$ defined on $\mfC\cap B_{2}\cap \{r>2\tau\}$ satisfying the bounds
\[
\begin{aligned}
\sup_{\mfC\cap B_{2}\cap \{r> 2\tau\}} r^{-1}|df|+|D^2f| &\leq \eta\\
\|\underline{f}\|_{L^2}(2, \tau) &\leq \delta.
 \end{aligned}
 \]
 \\
 \item[(ii)] We say that $N$ has the {\bf volume property} $P_2(\gamma)$ with respect to $\mfC$ if,  for each point $p \in (\mathbb{R}^k \times \{0\})\cap B_{2}$,   $N$ satisfies the volume bound
\[
\mathcal{H}^{n}(N\cap B_{2}(p)) \leq \mathcal{H}^{n}(\mfC\cap B_{2}(p)) +\omega_n(1-\gamma)^n
\]
\\
\item[(iii)] We say that $N$ has the {\bf harmonic property} $P_3(\delta)$ if, on $N$, the harmonic functions $\beta, y_i$ satisfy
\[
\|\underline{\beta}, \underline{y}\|_{L^{2}(N\cap B_4)} \leq \delta
\]
\end{itemize}
Finally, for $i=1,2,3$ we say that $N$ has property $P_i$ at scale $\rho$ if $\rho^{-1}N$ has property $P_i$.
\end{defn}

We now state the quantitative propagation of smallness estimate which will be a key component of the arguments to follow.  Roughly, the estimate says that if $N$ is the graph over $\mfC \cap B_{2} \cap \{r>2\tau\}$ of a small potential, then $N$ is graphical over $\mfC \cap B_1 \cap \{r> \tau_*\}$ for a quantifiable $\tau_* \leq \tau$ depending on $\|\underline{\beta}, \underline{y}\|_{L^2(N\cap B_{4})}$.  

\begin{prop}\label{prop: propOfSmall}
Fix $\gamma \in [\frac{1}{\sqrt{2}},1)$ and let $\mathcal{K}\subset \mathcal{M}$ be a compact set of smooth special Legendrians.  Let $\mfC = \mathbb{R}^k\times \Cone(\Sigma_{\kappa})$ for some $\kappa \in \cK$.  There exists constants $\delta_1, \delta_2, \eta_3, C$ depending only on $\mathcal{K}, \gamma$ with the following effect:

Suppose $N$ is an exact special Lagrangian in $\mathbb{C}^n$ satisfying
\begin{itemize}
\item The small graph property $P_1(\eta_3, 10^{-1}, \delta_1)$ at scale $\rho$ with respect to $\mfC$
\item The volume property $P_2(\gamma)$ at scale $\rho$ with respect to $\mfC$
\item The harmonic property and $P_3(\delta_2)$ at scale $\rho$.
\end{itemize} 
Define $\rho_*= (2-10^{-2})\gamma\rho$.  There exists $\tau_*>0$ satisfying the bound
\[
C\eta_3^{-1}\|\underline{\beta}, \underline{y}\|_{L^{2}(N\cap B_{4\rho})} \leq \tau_*^2 \leq 2C\eta_3^{-1}\|\underline{\beta}, \underline{y}\|_{L^{2}(N\cap B_{4\rho})}
\]
 such that $N$ is the graph associated to a local potential function $F$ extending $f$ over $ \mfC  \cap B_{\rho_*} \cap \{r>\rho\tau_*\}$ and such that the following estimates hold on $ \mfC  \cap B_{\rho_*} \cap \{r>\rho\tau_*\}$
 \[
  r^{-2}|F| \leq C\left(\left(\frac{r}{\rho}\right)^{-2}\|\underline{\beta}, \underline{y}\|_{L^{2}(N\cap B_{4\rho})} +\|\underline{f}\|_{L^{2}}(2\rho,10^{-1})\right),
  \]
  and
 \[
 r^{-1}|dF|+|D^2F| \leq C\left(\left(\frac{r}{\rho}\right)^{-2}\|\underline{\beta}, \underline{y}\|_{L^{2}(N \cap B_{4\rho})}+ \|\underline{f}\|_{L^{2}}(2\rho,10^{-1}))\right).
 \]
\end{prop}
\begin{proof}
By rescaling we may assume that $\rho=1$. Fix positive numbers $\gamma \in [\frac{1}{\sqrt{2}}, 1)$ and let $\mu \in (0,\frac{1}{10})$ be a constant to be determined.  To avoid carrying factors of $10$, we will set $\tau= \frac{1}{10}$. Fix $\eta_2$ to be the constant appearing in Lemma~\ref{lem: ellEst}, and let $\eta_1 = \eta_1(\cK, \eta_2, \tau, \frac{\mu}{100},\gamma)$ be the constant appearing in Proposition~\ref{prop: existPot1}.  We fix $\eta_3  <\min\{\eta_1, \eta_2\}$.   By volume monotonicity the volume bound assumption $P_2(\gamma)$ implies that 
\[
\mathcal{H}^{n}(N\cap B_{s}(p)) \leq \mathcal{H}^{n}(\mfC\cap B_{s}(p)) +\omega_ns^n(1-\gamma)^n
\]
for any point $p \in \mfC^{sing}\cap B_{2}(0)$, and any $s\leq2$; in particular, $N$ has $P_2(\gamma)$ at all scales $s\leq 2$.  Define sequences $\{s_{k}\}, \{a_{k}\}$ by
\[
\begin{aligned} 
a_k&=\gamma\mu^{k+1}\\
s_{k}&=s_{k-1}-a_{k-1} \quad s_0=\gamma
\end{aligned}
\]
The proof is based on the following two claims:
\\
\begin{itemize}
\item {\bf Claim 1}: Suppose that $f$ is defined on $\mfC \cap B_{s}\cap \{r>t\}$ for some $0<t<1<s<2$. Then, on this set we have
 \[
 r^{-2}|f| \leq C_{A}\left(r^{-2}\|\underline{\beta}, \underline{y}\|_{L^2(N\cap B_{4})} + \|\underline{f}\|_{L^{2}}(2,\tau)\right)
 \]
 for a constant $C_{A}$ depending only on $\cK$.
 \\
 \item {\bf Claim 2}: There is a $\delta >0$ depending only on $\cK, \gamma$ such that if $f$ is defined on $B_{2s_{k}} \cap \{r>2a_k\frac{\tau}{100}\}$ and satisfies
 \begin{equation}\label{eq: weakEst}
 \sup_{\mfC\cap B_{2s_{k}} \cap \{r>2a_k\frac{\tau}{100}\}} r^{-1}|df| + |D^2f| < \eta_{2}
 \end{equation}
 and
 \begin{equation}\label{eq: strongEst}
\sup_{\mfC\cap B_{2s_{k}} \cap \{r>2\mu^k\frac{\tau}{100}\}} r^{-2}|f| < \delta \eta_{3}
 \end{equation}
 then $f$ is defined on $\mfC \cap B_{2s_{k+1}}\cap\{r>2a_{k+1}\frac{\tau}{100}\}$ and satisfies
 \[
 \sup_{\mfC \cap B_{2s_{k+1}}\cap\{r>2a_{k+1}\frac{\tau}{100}\}} r^{-1}|df|+|D^2f| <\eta_2
 \]
 \end{itemize} 
 Let us first explain the proof of the the proposition, assuming the Claim 1 and Claim 2.  Let us first fix the constants.  Note that
 \[
 s_{k} = \gamma \left(1-\sum_{\ell=0}^{k}\mu^{\ell+1}\right) \geq \gamma\left(\frac{1-2\mu}{1-\mu}\right)
 \]
 thus for $\mu$ sufficiently small we can ensure that $2s_{k} \geq (2-10^{-2})$ for all $k$.  Fix $\delta_1 = \frac{\delta \eta_{3}}{2C_A}$ where $C_{A}, \delta$ are the constants appearing in Claims 1 and 2 respectively.  We now consider the following set
 \[
 \mathcal{I} = \{ k \in \mathbb{Z}_{\geq 0} : f \text{ exists on } \mfC\cap B_{2s_k} \cap \{r>2a_k \frac{\tau}{100}\} \text{ and satisfies~\eqref{eq: weakEst} } \}
 \]
 
 From our choice of $\eta_3$ and the volume assumption, Proposition~\ref{prop: existPot1} implies $0 \in \mathcal{I}$.  We claim that if $k \in \mathcal{I}$ and 
 \begin{equation}\label{eq: inductionEstPropSmall}
 (2a_{k}\frac{\tau}{100})^{-2}\|\underline{\beta}, \underline{y}\|_{L^2(N\cap B_{4})} \leq \frac{\delta \eta_{3}}{2C_{A}}
 \end{equation}
 then $k+1 \in \mathcal{I}$.  Indeed, if~\eqref{eq: inductionEstPropSmall} holds then by Claim 1 we have~\eqref{eq: strongEst} and hence by Claim 2, $k+1 \in \mathcal{I}$.  The second estimate in the statement follows from Lemma~\ref{lem: ellEst}.   Note that we may choose $\delta_2$ small so that~\eqref{eq: inductionEstPropSmall} holds for $k=0$.  This yields the proposition, assuming the claims.  
 \\
 
 \noindent {\bf Proof of Claim 1}:  First note that since $\beta^2, y_i^2$ are subharmonic by Lemma~\ref{lem: subharmonicFunctions}, the mean-value inequality (see e.g. \cite[Proposition 3.8]{CMbook} ) yields the bounds 
 \[
 \begin{aligned}
 \|\beta\|_{L^{\infty}(N\cap B_2)} &\leq C\|\underline{\beta}\|_{L^{2}(N\cap B_4)},\\
 \|y_i\|_{L^{\infty}(N\cap B_2)} &\leq C\|\underline{y}\|_{L^{2}(N\cap B_4)}
 \end{aligned}
 \]
for a constant $C$ depending only on $n$.  Fix a point $(0,z_1') \in \mfC\cap B_{s}\cap \{r>t\}$, and let $z'_0 = \frac{z'_1}{|z'_1|} $ which is a point in $\mfC \cap \{r=1\}$.  Integrating ~\eqref{eq: betaVsPot} along the curve from $(0, z_0')$ to $(0,z'_1)$ with tangent vector $\frac{\del}{\del R}$ yields
\[
\big||z'_1|^{-2}f(0, z'_1) - f(0, z'_0)\big| \leq \|\beta\|_{L^{\infty}(N \cap B_2)} \left(\frac{1}{|z'_1|^2} - 1 \right)
\]
and hence
\[
|z'_1|^{-2}|f(0, z'_1)| \leq \sup_{\mfC \cap \{(z,z'): |z'|=1, z=0\}}|f|+ Cr^{-2}\|\underline{\beta}\|_{L^{2}(N\cap B_4)}.
\]
Note that while $f$ is only locally defined, the curve lies in a single Darboux coordinate chart as constructed in Section~\ref{sec: Darboux}.  On the other hand by integrating the formula $y=- \frac{\del f}{\del x_i}$ (c.f. ~\eqref{eq: GraphCylinder}) along radial curves in $\mfC\cap \{z'=z'_1\}$ we get the bound in $ \mfC\cap B_s\cap\{r>t\}$ of the form
\begin{equation}\label{eq: preciseEstDecay}
r^{-2}|f| \leq  C_A (r^{-2}\|\underline{\beta}, \underline{y}\|_{L^{2}(N\cap B_{4})} +\|\underline{f}\|_{L^{2}}(2,10^{-1})) 
\end{equation}
for a uniform constant $C_A$ depending only on $\cK$.  Here we used  Lemma~\ref{lem: ellEst} to bound 
\[
\sup_{\mfC \cap \{(z,z'): z=0,\, |z'|=1\}}|f| \leq C\|\underline{f}\|_{L^{2}}(2,10^{-1})).
\]
Again, while $f$ is not globally defined, the integration takes place in a single Darboux coordinate chart.  This establishes the claim.
\\
\\
{\bf Proof of Claim 2}: Rescaling by $a_{k}^{-1}$ we obtain a potential for $a_{k}^{-1}N$ defined on $\mfC \cap B_{2s_ka_{k}^{-1}} \cap \{r>2\frac{\tau}{100}\}$ satisfying
\[
 \sup_{\mfC \cap B_{2s_ka_{k}^{-1}} \cap \{r>2\frac{\tau}{100}\}} r^{-1}|df| + |D^2f| < \eta_{2}
 \]
 and
 \[
\sup_{\mfC \cap B_{2s_ka_{k}^{-1}} \cap \{r>2\frac{\tau}{100}\}} r^{-2}|f| < \delta \eta_{3}
 \]
Covering this set by balls of radius $2$ and applying Lemma~\ref{lem: ellEst} we obtain
\[
 \sup_{\mfC \cap B_{2s_ka_{k}^{-1}-1} \cap \{r>2\tau\}} r^{-1}|df| + |D^2f| < C\delta \eta_{3}
 \]
 for a constant $C$ depending on $\cK, \gamma$.  Take $\delta= C^{-1}$.  Since the volume assumption $P_2(\gamma)$ holds, for each $p\in \mathbb{R}^k\times \{0\}$ the ball $B_{2}(p)$ satisfies the assumptions of Proposition~\ref{prop: existPot1}.  Therefore we obtain an extension of $f$ to $\mfC \cap B_{2\gamma}(p)\cap \{r >2\gamma \mu \frac{\tau}{100}\}$.  The locally defined normalized potentials glue to a well define normalized potential in each Darboux coordinate chart, yielding and extension of $f$ to $\mfC \cap B_{2s_ka_{k}^{-1}-2} \cap \{r>2\gamma \mu \frac{\tau}{100}\}$.  Scaling down by $a_k$ yields the claim.
\end{proof}

\section{Decay Estimates}\label{sec: Decay}
In this section we prove decay estimates which will eventually lead to the uniqueness of the tangent cone.  We begin with the following elementary result

\begin{lem}\label{lem: betaAndYRegEst}
Let $\cK \subset \cM$ be a compact set in the moduli space of special Legendrians and suppose that $\mfC= \mathbb{R}^k\times \Cone(\Sigma_\kappa)$ for some $\kappa \in \cK$.  Then there exists $\eta_4 = \eta_4(\cK, \tau)$ with the following effect: suppose $N$ is an exact special Lagrangian which is the graph associated to a potential $f$ defined on $\mfC \cap B_{2\rho} \cap \{r>2\rho\tau\}$ and $f$ satisfies the bounds
\[
\sup_{\mfC\cap B_2 \cap \{r>2\rho\tau \}}r^{-1}|df| + |D^2f| < \eta_{4}.
\]
Then we have the estimate
\[
\rho^{-n-4}\int_{N \cap B_{\rho} \cap \{r>4\rho\tau\}} \beta^2 + \sum_{i=1}^k\rho^{-n-2}\int_{N \cap B_{\rho} \cap \{r>4\rho\tau\}} y_i^2  \leq C\tau^{-2}  \|\underline{f}\|^2_{L^2}(2\rho,\tau)
\]
where $C= C(\cK)$ depends only on $\cK$.
\end{lem}
\begin{proof}
We only sketch the proof.  By rescaling we may assume that $\rho=1$.  Choosing $\eta_4$ small depending on $\tau$ we can ensure that $N\cap B_{1} \cap \{r>4\tau\}$ is part of the graph over $\mfC\cap B_2\cap \{r>2\tau\}$.  From the formulas~\eqref{eq: betaVsPot} and~\eqref{eq: GraphCylinder}  relating for $\beta, y_i$ and the potential $f$, together with Lemma~\ref{lem: ellEst} the result follows.
\end{proof}

We now state a decay lemma in the easiest case.

\begin{lem}\label{lem: easyDecay}
Fix $\rho>0$, $\tau \in (0,\frac{1}{10}]$.  Let $\mathcal{K}\subset \cM$ be a compact set in the moduli space of smooth special Legendrians.  Suppose $N$ is an exact special Lagrangian which is the graph of a potential function $f$ over $(\mathbb{R}^k\times \Cone(\Sigma_{\kappa})) \cap B_{2\rho}\cap \{r>2\rho\tau\}$ for some $\kappa \in \cK$ .  Let $\eta_4= \eta_{4}(\cK, \tau)$ be the constant appearing in Lemma~\ref{lem: betaAndYRegEst}.  Then there exists a constant $\delta_3= \delta_3(\cK)>0$ such that if the potential function $f$ satisfies
\[
\sup_{\mfC\cap B_{2\rho}\cap \{r>2\rho\tau\}}r^{-1}|df| + |D^2f| < \eta_4,
\]
and if
\begin{itemize}
\item[$(i)$]
\[
\|\underline{f}\|_{L^{2}}(2\rho, \tau) \leq \delta_3\tau^2 10^{-3} \|\underline{\beta}, \underline{y}\|_{L^{2}(N\cap B_{4\rho})}
\]
\item[$(ii)$]
\[
\rho^{-n}\cH^n(N\cap B_{\rho}\cap\{r<4\rho \tau\}) \leq \omega_n2^{-(n+8)}10^{-6}
\]
\end{itemize}
then the following estimate holds:
\[
\begin{aligned}
\|\underline{\beta}, \underline{y}\|_{L^{2}(N\cap B_{\rho})} &\leq \frac{1}{100}\|\underline{\beta}, \underline{y}\|_{L^{2}(N\cap B_{4\rho})}.
\end{aligned}
\]
\end{lem}

\begin{proof}
We prove the statement for $\beta$ only, with the statement for $y_i$ be obtained by the same argument.  By rescaling we may assume that $\rho=1$. By the mean value inequality applied to $\beta^2$, which is subharmonic by Lemma~\ref{lem: subharmonicFunctions}, we have
\[
\|\beta\|^2_{L^{\infty}(N\cap B_{1} \cap \{r<4\tau\})} \leq \frac{2^{n+4}}{\omega_n} \|\underline{\beta}\|^2_{L^{2}(N\cap B_4)}
\]
and so
\[
\begin{aligned}
\|\beta\|_{L^{2}(N\cap B_1\cap \{r< 4\tau\})} &\leq \|\beta\|_{L^{\infty}(N\cap B_{1} \cap \{r<4\tau\})}\left(\cH^n(N\cap B_{1}\cap\{r<4\tau\})\right)^{\frac{1}{2}}\\
&\leq 10^{-3} \|\underline{\beta}\|_{L^{2}(N\cap B_4)}.
\end{aligned}
\]
By our choice of $\eta_4$,  Lemma~\ref{lem: betaAndYRegEst} yields
\[
\|\beta\|_{L^{2}(N\cap B_{1}\cap \{r>4\tau\})} \leq C\tau^{-2}\|\underline{f}\|_{L^{2}}(2,\tau)
\]
for $C$ depending only on $\cK$.  Taking $\delta_3 = C^{-1}$ and applying $(i)$ yields the result.
\end{proof}

Recall the normalized volume excess defined in~\eqref{eq: volExDefn}, and note that ${\rm VolEx}_{N}(r)$ is invariant under the action of $SU(n)$.  For $r_1>r_2$ we will denote 
\[
{\rm VolEx}_{N}(r_1,r_2) = {\rm VolEx}_{N}(r_1)-{\rm VolEx}_{N}(r_2) \geq 0
\]
 The next proposition establishes a quantitative estimate for the volume excess when the tangent cone $\mfC$ has singularities in codimension at least $5$, leading eventually to Theorem~\ref{thm: main2}.  The idea of the proof is that, if $N$ is graphical over $\mfC$ on $B_2\cap\{r>2\tau\}$, then $\cH^{n}(N\cap B_{2}\cap \{r>2\tau\})$ is essentially equal to $\cH^{n}(\mfC\cap B_{2}\cap \{r>2\tau\})$, up to small error terms.  The remaining region $N\cap B_2\cap\{r<2\tau\}$ can be controlled by a comparison argument, using that $N$ is volume minimizing.
 
\begin{prop}\label{prop: decayOfVolEx}
Let $\mathcal{K} \subset \mathcal{M}$ be a compact set of smooth special Legendrians and suppose that $\mfC = \mathbb{R}^k \times \Cone(\Sigma_{\kappa})$ for some $\kappa \in \mathcal{K}$.  Suppose that $n-k \geq 5$.  If $N \subset \mathbb{C}^n$ is an exact special Lagrangian satisfying
\begin{itemize}
\item The small graph property  $P_1(\eta, 10^{-1}, \delta_1)$ at scale $\rho$ with respect to $\mfC$, 
\item The volume property $P_2(\frac{5}{6})$ at scale $\rho$ with respect to $\mfC$
\item The harmonic property $P_3(\delta_2)$ at scale $\rho$ with respect to $\mfC$,
\end{itemize}
where $\delta_1(\mathcal{K}), \delta_2(\mathcal{K}),  \eta= \eta_3(\mathcal{K}, \frac{5}{6})$ are the constants of Proposition~\ref{prop: propOfSmall}. Then there is a constant $C$ depending only on $\mathcal{K}$ such that
\[
{\rm VolEx}_{N}(\rho) \leq C\left(\|\underline{f}\|^2_{L^{2}}(2\rho, 10^{-1}) + \| \underline{\beta}, \underline{y}\|^2_{L^{2}(N\cap B_{4\rho})} + {\rm VolEx}_{N}(2\rho)^{\frac{n}{n-1}}\right).
\]
\end{prop}
\begin{proof}
By rescaling, it suffices to assume $\rho=1$. To ease notation, let us denote
\[
M_1:= \|\underline{f}\|_{L^{2}}(2, 10^{-1}) \qquad M_2 := \| \underline{\beta}, \underline{y}\|_{L^{2}(N\cap B_{4})}.
\]
Since the assumptions of Proposition~\ref{prop: propOfSmall} are in effect for $\gamma = \frac{5}{6}$ we can assume that $N$ is the graph of a potential on $\mfC\cap B_{\frac{3}{2}}\cap \{r> \frac{\tau_0}{10}\}$ for
\begin{equation}\label{eq: tauVolEst}
C^{-1}M_2 \leq \tau_0^2 \leq CM_2.
\end{equation}
We first analyze the portion of $N$ which is obtained from the graph of $df$.  From Proposition~\ref{prop: propOfSmall} we get estimate
\begin{equation}\label{eq: locVolDecayRegEst}
r^{-1}|df| +|D^2f| \leq C\left(r^{-2}\| \underline{\beta}, \underline{y}\|_{L^{2}(N\cap B_{4})} + \|\underline{f}\|_{L^{2}}(2, 10^{-1})\right)
\end{equation}
on $\mfC \cap B_{\frac{3}{2}}\cap\{r>\tau_0\}$, and hence, from the bound for $\tau_0$ we have
\[
\begin{aligned}
&\sup_{\mfC\cap B_{\frac{3}{2}} \cap \{r>1\}}|df| < C(M_1+M_2)\\
&\sup_{\mfC\cap B_{\frac{3}{2}} \cap \{ r>\tau_0 \}}|df|< C\tau_0
\end{aligned}
\]
Since the vectors $\frac{\del}{\del R}, \frac{\del}{\del r}$ are tangent to $\mfC$ it follows that the graph of $df$ over $\mfC \cap B_{\frac{3}{2}}\cap\{r>\tau_0\}$ necessarily contains all points in $N\cap B_{\frac{3}{2}-\epsilon} \cap \{r>\tau_0+C\tau^2_0\}$ where
\[
\epsilon < C(M_1^2+M_2^2)
\]
and $C$ depends only on $\cK$.  Set $\tau = \tau_0+C\tau^2_0$ for simplicity.  For the remainder of the proof $C$ will denote a constant which can increase from line to line, but depends only on $\cK$.  

Since $\mfC= \mathbb{R}^{k}\times \bfC$ where $\bfC$ is a cone with $\dim \bfC=n-k$ we have
\begin{equation}\label{eq: coneVolEst}
\mathcal{H}^{n}(\mfC \cap B_2 \cap \{r \leq 10\tau\}) \leq C\tau^{n-k}.
\end{equation}
Using~\eqref{eq: locVolDecayRegEst} we see that, for any $s\leq \frac{3}{2}-\epsilon$ we have
\begin{equation}\label{eq: stupidVolCor}
\begin{aligned}
&\bigg|\mathcal{H}^{n}(N\cap B_s\cap \{r>\frac{\tau}{10}\}) - \mathcal{H}^{n}(\mfC\cap B_s\cap \{r>\frac{\tau}{10}\})\bigg|\\
& \quad\leq C\left(\epsilon + \int_{\mfC\cap B_s \cap \{r>\frac{\tau}{10}\}} M_1^2+r^{-4}M_{2}^2 + \tau^{n-k}\right).
\end{aligned}
\end{equation}
Since $r^{m-5}dr$ is integrable near $0$ for $m \geq 5$, we get a bound
\begin{equation}\label{eq: regVolEst}
\bigg|\mathcal{H}^{n}(N\cap B_s\cap \{r>\frac{\tau}{10}\}) - \mathcal{H}^{n}(\mfC\cap B_s\cap \{r>\frac{\tau}{10}\})\bigg|\leq C (M_1^2+M_{2}^2),
\end{equation}
for any $s\leq \frac{3}{2}-\epsilon$.  We now bound the volume of $N$ from above.  Since $N$ is special Lagrangian, and hence volume minimizing, it suffices to construct a competitor surface.

Write $N$ as the graph over $\mfC\cap B_2 \cap \{r>\frac{\tau}{10}\}$ of a normal vector field $V(z,z')$. Define 
\[
A = \begin{cases} \text{ graph}^{\perp}\left( V(z,z')\right) & \text{ if }(z,z') \in \mfC\cap B_2 \cap \{r >\frac{\tau}{10}\}\\
\text{ graph}^{\perp}\frac{10r}{\tau} V(z,  \frac{\tau z'}{|z'|10}) & \text{ if } (z,z') \in \mfC\cap B_2 \cap \{r \leq \frac{\tau}{10}\}
\end{cases}
\]
where ${\rm graph}^{\perp}$ denotes the normal graph.  Note that $A$ can be smoothed near $r=\frac{\tau}{10}$ if desired.  By construction, the current $N-A$ is supported in $B_2\cap \{r\leq \frac{\tau}{10}\}$ and for any $s \leq \frac{3}{2}-\epsilon$;
\begin{equation}\label{eq: N-AMass}
\begin{aligned}
{\rm Mass}_{B_s}(N-A) &\leq C\tau^{n-k} + \mathcal{H}^n(N\cap B_s\cap \{r < \frac{\tau}{10}\})\\
&\leq C(M_1^2+M_2^2)+ \mathcal{H}^n(N\cap B_s)- \mathcal{H}^{n}(\mfC \cap B_s) 
\end{aligned}
\end{equation}
where we used~\eqref{eq: tauVolEst},~\eqref{eq: coneVolEst} and~\eqref{eq: regVolEst} and $n-k\geq 5$.

Applying the co-area formula \cite[Lemma 28.1]{SimonBook}, together with the slicing theory for integer multiplicity currents \cite[Lemma 28.5]{SimonBook} we can find an $s_*\in [\frac{5}{4},1]$ and $n-1$ dimensional integer multiplicity current $T$ such that $\del T=0$ such that
\[
T= \del \left((N-A)\big|_{B_{s_*}}\right)
\]
and
\[
{\rm Mass}(T) \leq C\left({\rm Mass}_{B_{\frac{5}{4}}}(N-A) - {\rm Mass}_{B_1}(N-A) \right)
\]
By the isoperimetric inequality of Federer-Fleming \cite{FF} (see also \cite[Theorem 30.1]{SimonBook}) there exists a current $P$, with $\del P= T$ and such that
\[
{\rm Mass}(P) \leq c(n) ({\rm Mass}(T))^{\frac{n}{n-1}}
\]
for a dimensional constant $c(n)$.  Now we have
\[
\del\left(N\big|_{B_{s_*}}\right) = \del \left(A\big|_{B_{s_*}} + P\right).
\]
Thus, by the volume minimization property of special Lagrangians
\[
\begin{aligned}
\mathcal{H}^n(N\cap B_{s_*})&\leq {\rm Mass}_{B_{s_*}}(A) + {\rm Mass}(P)\\
& \leq C\tau^{n-k} + \mathcal{H}^n(N\cap B_{s_*} \cap \{r>\frac{\tau}{10}\})\\
&\quad + c(n)\left({\rm Mass}_{B_{\frac{5}{4}}}(N-A) - {\rm Mass}_{B_1}(N-A) \right)^{\frac{n}{n-1}}\\
&\leq C\tau^{n-k} + \mathcal{H}^n(\mfC\cap B_{s_*}) + C(M_1^2+M_2^2)\\
&\quad+C\left(M_1^2+M_2^2+ \mathcal{H}^n(N\cap B_{\frac{5}{4}})- \mathcal{H}^{n}(\mfC \cap B_{\frac{5}{4}}) \right)^{\frac{n}{n-1}}
 \end{aligned}
 \]
 where we used~\eqref{eq: regVolEst} and~\eqref{eq: N-AMass}.  Thus, from~\eqref{eq: tauVolEst} and~\eqref{eq: coneVolEst}, we obtain
 \[
 {\rm VolEx}(s_*) \leq C\left(M_1^2+M_2^2 + {\rm VolEx}\left(\frac{5}{4}\right)^{\frac{n}{n-1}}\right)
 \]
 for a constant $C$ depending only on $\cK$. Volume monotonicity, together with $1 \leq s_* \leq \frac{5}{4} \leq 2$ yields the desired result.
\end{proof}

Note that the proof of Proposition~\ref{prop: decayOfVolEx} yields the following simple corollary

\begin{cor}\label{cor: easyVolEst}
Let $\mathcal{K} \subset \mathcal{M}$ be a compact set of smooth special Legendrians and suppose that $\mfC = \mathbb{R}^k \times \Cone(\Sigma_{\kappa})$ for some $\kappa \in \mathcal{K}$.  If $N \subset \mathbb{C}^n$ is an exact special Lagrangian satisfying
\begin{itemize}
\item The small graph property  $P_1(\eta, 10^{-1}, \delta_1)$ at scale $\rho$ with respect to $\mfC$, 
\item The volume property $P_2(\frac{5}{6})$ at scale $\rho$ with respect to $\mfC$
\item The harmonic property $P_3(\delta_2)$ at scale $\rho$ with respect to $\mfC$,
\end{itemize}
where $\delta_1(\mathcal{K}), \delta_2(\mathcal{K}),  \eta= \eta_3(\mathcal{K}, \frac{5}{6})$ are the constants of Proposition~\ref{prop: propOfSmall}. Then there is a constant $C$ depending only on $\mathcal{K}$ such that for any $\tau$ satisfying  $\tau^2 > C\|\underline{\beta},\underline{y}\|_{L^{2}(N\cap B_{4\rho})}$ we have
\[
\begin{aligned}
&\bigg|\mathcal{H}^{n}(N\cap B_\rho\cap \{r>\tau\}) - \mathcal{H}^{n}(\mfC\cap B_s\cap \{r>\tau\})\bigg|\\
& \leq C\left(\|\underline{f}\|^2_{L^2}(2\rho, 10^{-1}) +\tau^{n-k} + \left(\int_{\tau}^{1}r^{n-k-5}dr\right)\|\underline{\beta},\underline{y}\|^2_{L^{2}(N\cap B_{4\rho})}\right)
\end{aligned}
\]
\end{cor}
\begin{proof}
This follows from the arguments in the proof of Proposition~\ref{prop: decayOfVolEx} leading to~\eqref{eq: stupidVolCor}.
\end{proof}

We now come to the main decay result.  Roughly speaking, this result says that if at some scale the special Lagrangian $N$ is well modeled by the linearized special Lagrangian equation on $\mfC$, then $N$ gets significantly closer to some possibly different special Lagrangian cylinder $\mfC'$ after passing to a smaller scale.  The cylinder $\mfC'$ is obtained from $\mfC$ by deformation using crucially the assumption of integrability.  

\begin{prop}\label{prop: fDecay}
Fix $\rho>0$ and $\tau, \eta \in (0,\frac{1}{10})$, constants $C_1, C_2$ and a compact set $\mathcal{K}\subset \mathcal{M}$.  Let $\delta_2(\cK)$ be the constant defined in Proposition~\ref{prop: propOfSmall}.  There exists $\theta(\mathcal{K}) \in (0,\frac{1}{2})$ and constants $\delta_4 =\delta_4(\mathcal{K}, \tau, \eta, C_1, C_2)$, $\eta_5(\mathcal{K})$, $b_0=b_0(\mathcal{K},\tau) \in \mathbb{Z}_{>0}$ with the following effect:  if $N$ is an exact special Lagrangian with harmonic function $\beta$ normalized to satisfy $Av_{N}(\beta, 4\rho) =0$ and with respect to this normalization $N$ satisfies \begin{itemize}
\item[$(i)$] there is some $\kappa \in \cK$ such that, with respect to $\mfC:= \mathbb{R}^k\times \Cone(\Sigma_{\kappa})$, and at scale $\rho$, $N$ has:
\begin{itemize}
\item the small graph property $P_1(\eta_5, \tau, \delta_{4})$, 
\item the volume property $P_2(\frac{5}{6})$ 
\item the harmonic property $P_{3}(\delta_2)$ 
\end{itemize} 
\item[$(ii)$] The harmonic functions $\beta, y_i$ satisfy 
\[
\|\underline{\beta}, \underline{y}\|_{L^{2}(N\cap B_{4\rho})} \leq C_1\tau^{-2}\|\underline{f}\|_{L^{2}}(2\rho, \tau)
\]
\item[$(iii)$] For some $b_0\leq b \leq +\infty $ the volume excess satisfies the inequality
\[
{\rm VolEx}_{N}(2\rho, 2^{1-b}\rho) \leq C_2\|\underline{f}\|_{L^{2}}(2\rho, \tau).
\]
 \end{itemize}
 Then, there are constants $C(\cK, \tau) >0$ and $c(\cK, \tau, b)\leq 1$ with $c(\cK, \tau,\infty)=0$, such that the following conclusions hold: there are elements $a \in \mathfrak{su}(n)$, and $q \in \mfH_{2m}(\Sigma_{\kappa}) $ satisfying the bounds
 \[
 |a|_{\mathfrak{su}(n)}+\|q\|_{L^{2}(\Sigma_{\kappa})}<C\|\underline{f}\|_{L^2}(2\rho, \tau)
 \]
 such that if we set $\widehat{N} := \exp(-a)N$ and $\widehat{\mfC} = \mathbb{R}^k\times \Cone(\exp^\cM_{\Sigma_{\kappa}}q)$, then
 \begin{itemize}
 \item[${\bf (c_1)}$] we have the decay estimates
 \[
 \begin{aligned}
 \|\underline{y}\|^2_{L^{2}(\widehat{N}\cap B_{2\theta\rho})}& \leq c(b) {\rm VolEx}_{N}(2\rho, 2^{1-b}\rho) +  \frac{1}{10} \|\underline{y}\|^2_{L^{2}(\widehat{N}\cap B_{4\rho})}\\
 \|\underline{\beta-Av_{\widehat{N}}(\beta, 2\theta\rho)}\|^2_{L^{2}(\widehat{N}\cap B_{2\theta\rho})} &\leq c(b) {\rm VolEx}_{N}(2\rho, 2^{1-b}\rho) +  \frac{1}{10} \|\underline{\beta}\|^2_{L^{2}(\widehat{N}\cap B_{4\rho})}
 \end{aligned}
\]
\item[${\bf (c_2)}$] $\widehat{N}$ is the graph associated to a potential function $\hat{f}$ defined in an open neighborhood of $\widehat{\mfC} \cap B_{\theta\rho} \cap \{r> \theta\rho\tau\}$ 
\[
\|\underline{\hat{f}-Av_{\widehat{N}}(\beta, 2\theta\rho)}\|^2_{L^2}(\rho\theta, \tau ) \leq c(b){\rm VolEx}_{N}(2\rho, 2^{1-b}\rho) +\frac{1}{10}\|\underline{f}\|^2_{L^2}(2\rho, \tau). 
\]
and 
\[
\sup_{\widehat{\mfC} \cap B_{2\theta\rho} \cap \{r> 2\theta\rho\tau\}} r^{-1}|d\hat{f}| + |D^2\hat{f}| \leq \eta.
\]
\end{itemize}
 \end{prop}
\begin{proof}
The proof is by contradiction and compactness.  By rescaling we may assume $\rho=1$.  Fix $\eta, \tau \in (0,\frac{1}{10})$, $b_0\in \mathbb{Z}_{>0}$ to be determined, and let $\eta_5 = \frac{1}{10}\eta_3(\cK, \frac{5}{6})$, where $\eta_3(\cK, \frac{5}{6})$ is the constant appearing in Proposition~\ref{prop: propOfSmall}.  Suppose $N_\nu$ is a sequence of exact special Lagrangians satisfying $(i)-(iii)$, $P_1(\eta_5, \tau, \delta_4(N_\nu))$, $P_2(\frac{5}{6})$, $P_3(\delta_2)$ with respect to $\mfC_{\kappa_{\nu}} = \mathbb{R}^k \times \Cone(\Sigma_{\kappa_{\nu}})$ for $\kappa_{\nu}\in \mathcal{K}$, $\eta \leq \eta_5$, and $\delta_{\nu}= \delta_4(N_\nu)\rightarrow 0$.  Without loss of generality we may assume that $\mfC_{k_{\nu}}$ converge to $\mfC$ as closed integral currents, and smoothly on compact sets away from $\mathbb{R}^{k}\times \{0\}$.  Proposition~\ref{prop: propOfSmall} implies that the potential $f_\nu$  of $N_{\nu}$ is defined on $\mfC\cap B_{\frac{3}{2}} \cap \{r > \tau_\nu\}$ with $\tau_\nu\rightarrow 0$ and satisfies the estimates 
 \begin{equation}\label{eq: ellipticEstimates}
 \begin{aligned}
 |f_{\nu}| &\leq C\tau^{-2}\|\underline{f_{\nu}}\|_{L^{2}}(2,\tau) (1 +r^{2})\\
r^{-1}|df_\nu| + |D^2f_\nu| &\leq C\tau^{-2}\|\underline{f_{\nu}}\|_{L^{2}}(2,\tau) (1 +r^{-2}) \\
 \text{ for all } (z,z') &\in \mfC\cap B_{\frac{3}{2}}\cap \{r>\tau_\nu\}.
 \end{aligned}
 \end{equation}
for $C= C(\cK)$.  Dividing through by $\|\underline{f_{\nu}}\|_{L^{2}}(2,\tau)$ we obtain a sequence of locally defined potentials $\tilde{f}_\nu$ converging locally smoothly (along a subsequence) to a well-defined function $\tilde{f}$ by Lemma~\ref{lem: diffPotEst}.  Furthermore, $\tilde{f}$ satisfies $\Delta \tilde{f}=0$.  From~\eqref{eq: ellipticEstimates}, the locally defined functions $\tilde{f}_\nu$ satisfy
 \[
 |d\tilde{f}_\nu| \leq Cr^{-1}, \quad |\tilde{f}_\nu| \leq C
 \]
 for a constant $C$ depending only on $\cK, \tau$. Since $r^{-2}$ is integrable (as $\mfC = \mathbb{R}^k \times \bfC$ with $\dim \bfC >2$ by assumption) this yields $\tilde{f} \in W_{loc}^{1,2} (B_{\frac{3}{2}}\cap (\mfC\setminus \mfC_{sing}))$ and we have the estimate
 \begin{equation}\label{eq: W12bndLimit}
 \int_{\mfC \cap B_\frac{3}{2}} |d\tilde{f}|^2 + r^{-2}|\tilde{f}|^2 <C
 \end{equation}
  for $C$ depending only on $\cK, \tau$. Let $\bfC= \Cone(\Sigma)$ and denote by
 \[
 \lambda_0 = 0 < \lambda_1 < \lambda_2 < \cdots < \lambda_j < \cdots
 \]
 the distinct eigenvalues of $\Delta_{\Sigma}$, and, for each $j$ let 
 \[
 \mathfrak{H}_{\lambda_j}(\Sigma) := \{ \phi:\Sigma \rightarrow \mathbb{R} : \Delta_{\Sigma}\phi+\lambda_j\phi =0\}.
 \]
 Fix $\phi^j \in \mathfrak{H}_{\lambda_j}(\Sigma)$ and define
 \[
 v^j(x,r) := \langle \phi^j, \tilde{f}(x,r)\rangle_{L^{2}(\Sigma)}.
 \]
 Then since $\Delta_{\mfC} \tilde{f}=0$ we obtain
 \[
 \Delta_{\mathbb{R}^k} v^j + \frac{\del}{\del r^2}v^j + \frac{n-k-1}{r}\frac{\del}{\del r} v^{j} - \lambda_jr^{-2}v^j =0.
 \]
 Set $\tilde{v}^j = r^{\alpha_j}v^{j}$ for $\alpha_j$ to be determined, and let us momentarily suppress the dependence on $j$ to ease the notation.  Let $\alpha$ be the negative root of $\alpha^2-(n-k-2)\alpha-\lambda=0$ and choose $\chi=\sqrt{(n-k-2)^2+ 4\lambda}>0$. 

Then by direct calculation we have
\[
\Delta_{\mathbb{R}^k} \tilde{v} +  \frac{1}{r^{1+\chi}} \frac{\del}{\del r}\left( r^{1+\chi}\frac{\del \tilde{v}}{\del r}\right)=0.
\]
From the bound~\eqref{eq: W12bndLimit} we get, for any $\rho < \frac{3}{2}$
\[
\int_{\{r^2+|z|^2< \rho^2\}}\int_{r<\rho} (|\nabla \tilde{v}^j| + r^{-2}\tilde{v}^j) r^{1+\chi_j} drdz <+\infty.
\]
By Simon's real analyticity of Fourier series \cite[Appendix 1]{SimonUn} we deduce that $\tilde{v}^{j}$ is real analytic in the variables $r^2, x$.  That is 
\[
\tilde{v}^{j} = \sum_{p=0}^{\infty} \sum_{\ell \in \mathbb{Z}_{\geq 0}^{k}} a^j_{p,\ell} r^{2p}x^{\ell}, \quad \text{ for } \quad r^2+|x|^2 \leq (\theta_0\rho)^2
\]
where $\ell = (\ell_1, \ldots, \ell_k) \in \mathbb{Z}_{\geq 0}^{k}$ is a multi-index, $|\ell| = \ell_1 + \cdots + \ell_k$,  $\theta_0 \in (0,1)$ is a fixed constant depending only on $\chi_j$, and
\[
|a^j_{p,\ell}| \leq (\theta_0\rho)^{-2p-|\ell|} \left(\rho^{-\chi_j-3}\int_{r^2+|x|^2 <\rho^2} (\tilde{v}^{j})^2 \right)^{\frac{1}{2}}.
\]
We shall write $|\ell| = \ell_1 + \cdots + \ell_k$.    Substituting this expression into the formula for $\tilde{f}$ we have
\begin{equation}\label{eq: expOfFinfty}
\tilde{f} = \sum_{\{(p, \ell, \alpha_j): 2p+|\ell|-\alpha_j \leq 2\}}\sum_{i=1}^{N_j}a_{p,\ell,i}^jr^{2p-\alpha_j}x^{\ell}\phi^j_{i} + \tilde{f}_{>2}
\end{equation}
where $\phi^j_{1},\ldots, \phi^j_{N_{j}}$ is an orthonormal basis of $\mfH_{\lambda_j}(\Sigma)$ with $\lambda_j$ corresponding to $\alpha_j$ as above and $\tilde{f}_{>2}$ is harmonic with strictly faster than quadratic growth.  Terms with $2p + |\ell| -\alpha_j <0$ are ruled out since $p \geq 0$, $|\ell| \geq 0$, and $\alpha_j \leq 0$.  

Next we shall make use of the assumption on the volume excess to control the terms of total degree $2p + |\ell| -\alpha_j <2$.  If such a term occurs then we must have $p=0$ and $|\ell|=0,1$.  From the volume monotonicity formula we have
\[
{\rm VolEx}_{N_{\nu}}(s) = \int_{N_{\nu}\cap B_{s}} \frac{|x^{\perp}|^2}{R^{n+2}}.
\]
On the other hand, for any exact Lagrangian $N$ we have $\nabla^{N}\beta= J(x^{\perp})$ and hence
\[
{\rm VolEx}_{N_{\nu}}(s) =\int_{N_{\nu}\cap B_{s}} \frac{|d\beta_{\nu}|^2}{R^{n+2}}.
\]
Define
\[
\tilde{\beta}_\nu = \frac{\beta_\nu}{\|\underline{f_{\nu}}\|_{L^{2}}(2,\tau)}, \qquad \tilde{y}_{i, \nu} = \frac{y_i}{\|\underline{f_{\nu}}\|_{L^{2}}(2,\tau)} .
\]
where we regard $y_i$ as functions on $N_\nu$. From the bound $(ii)$ we see that $\|\underline{\tilde{\beta}_\nu},\underline{\tilde{y}_{i,\nu}}\|_{L^{2}(N_\nu\cap B_4)} \leq C_1\tau^{-2}$.  Since $\tilde{\beta}_\nu, \tilde{y}_{i,\nu}$ are harmonic, the mean-value inequality yields $L^{\infty}$ bounds on compact sets of $N_{\nu}\cap B_4$.  Thus $\tilde{\beta}_\nu, \tilde{y}_{i,\nu}$ converge locally smoothly to harmonic functions $\tilde{\beta}, \tilde{y}_i$  on compact subsets of $\mfC\setminus \mfC_{sing} \cap B_4$. From the $L^{\infty}$ bound the convergence holds in $L^2(\mfC \cap B_3)$.  From the formulas for $\beta, y_i$ in terms of $f$, together with~\eqref{eq: betaVsPot},~\eqref{eq: GraphCylinder}, we have
\begin{equation}\label{eq: relFBetaY}
\tilde{\beta} = -\frac{1}{2}R^3 \frac{\del}{\del R} \left(\frac{\tilde{f}}{R^2}\right), \qquad \tilde{y}^i= -\frac{\del \tilde{f}}{\del x^i}.
\end{equation}
In particular,
\begin{equation}\label{eq: relFBetaYSeries}
\begin{aligned}
\tilde{\beta} &=-\frac{1}{2} \sum_{\{(p, \ell, \alpha_j): 2p+|\ell|-\alpha_j \leq 2\}}\sum_{i=1}^{N_j}(2p+|\ell|-\alpha_j -2)a_{p,\ell,i}^j r^{2p-\alpha_j}x^{\ell}\phi^j_{i} + \tilde{\beta}_{>2}\\
\tilde{y}_k&= - \sum_{\{(p, \ell, \alpha_j): 2p+|\ell|-\alpha_j \leq 2\}}\sum_{i=1}^{N_j}\ell_ka_{p,\ell,i}^j r^{2p-\alpha_j}x^{\ell-e_k}\phi^j_{i} + \tilde{y_k}_{>1}
\end{aligned}
\end{equation}
where $e_k$ is the standard unit vector with $1$ in the $k$-th slot,  $\tilde{\beta}_{>2} $ is a harmonic function with strictly larger than quadratic growth and $\tilde{y_k}_{>1}$ is harmonic with strictly larger than linear growth.  Since $\tilde{\beta}_\nu$ is harmonic and satisfies $Av_{N_{\nu}}(\tilde{\beta}_\nu, 4)=0$  we see that $Av_{\mfC}(\tilde{\beta}, 4)=0$.  From the formula for $\tilde{\beta}$ in terms of $\tilde{f}$, this immediately yields that the expansion~\eqref{eq: expOfFinfty} does not contain a term of degree $0$.  Define
\[
\begin{aligned}
I_{(0,k)} &= \{(j, p, \ell) : 2p+|\ell| -\alpha_j \in (0,k)\},\\
I_k&= \{(j, p, \ell) : 2p+|\ell| -\alpha_j =k\}.\\
\end{aligned}
\]
Then we can write
\[
\tilde{f} = \tilde{f}_{(0,2)} +\tilde{f}_2 +\tilde{f}_{>2}
\]
where, for example, $\tilde{f}_{(0,2)}$ is obtained by summing over  $(j,p,\ell) \in I_{(0,2)}$,  and $\tilde{f}_2$ is defined analogously using $I_2$.  Similarly we can write
\[
\tilde{\beta} = \tilde{\beta}_{(0,2)} + \tilde{\beta}_{>2} \qquad \tilde{y}_i = (\tilde{y}_i)_{(0,1)} + (\tilde{y}_{i})_{1} + (\tilde{y}_i)_{>1}.
\]

We need to exploit the volume excess control to estimate terms arising from $\tilde{f}_{(0,2)}$.  Suppose $\tilde{f}$ contains a term of total homogeneous degree $0<2p+|\ell|-\alpha_j<2$.  Since $\alpha_j <0$ the number of such terms is finite and bounded uniformly in $\cK$.  Furthermore, from assumption $(iii)$ on the volume excess and Fatou's lemma we have
 \begin{equation}\label{eq: volExintegral}
 \int_{\mfC \cap \{2^{1-b} < R < 2\}} R^{-(n+2)} |\nabla \tilde{\beta}|^2  <C_2<+\infty.
 \end{equation}
 In particular, for $\nu$ sufficiently large Fatou's lemma gives
 \[
  \int_{\mfC \cap \{2^{1-b} < R < 2\}} R^{-(n+2)} |\nabla \tilde{\beta}|^2\leq 5  \int_{N_{\nu} \cap \{2^{1-b} < R < 2\}} R^{-(n+2)} |\nabla \tilde{\beta}_{\nu}|^2.
  \]
 Before proceeding note that if $b=+\infty$, then the bound~\eqref{eq: volExintegral} together with~\eqref{eq: relFBetaYSeries} implies that expansion of $\tilde{f}$ does not contain any terms with $0<2p+|\ell|-\alpha_j<2$.  What follows is an effective version of this result for sufficiently large annuli.  Indeed, from ~\eqref{eq: relFBetaYSeries} that if $0<2p+|\ell|-\alpha_j <2$, then there is a constant $C_3= C_3(\cK)$ such that
 \begin{equation}\label{eq: lowModeBound}
 |a^j_{p,\ell,q}|^2\cdot2^{-2b(2p+|\ell|-\alpha_j-2)} \leq  C_3 \int_{\mfC\cap \{2^{1-b} < R < 2\}} R^{-(n+2)} \bigg|\frac{\del}{\del R} \left(R^3 \frac{\del}{\del R} \left( \frac{\tilde{f}}{R^2}\right)\right)\bigg|^2.
 \end{equation}
Fix $\theta \in (0, \frac{1}{2})$ with $\theta <\theta_0$.  Since $\mathfrak{H}_{2m}(\Sigma_{\kappa})$ has the same dimension for all $\kappa \in \cK$ we have $2p+|\ell|-\alpha_j <2-\delta(\cK)$ for some $\delta(\cK)$ depending only on $\cK$.  Thus we can choose $b$ sufficiently large depending only on $\cK, \theta$ such that the lower degree terms are negligible
\begin{equation}\label{eq: flowDegreeDecay}
\theta^{-n-4}\int_{\mfC\cap B_{\theta}}|\tilde{f}_{(0,2)}|^2 \leq 10^{-2} \int_{\mfC \cap\{2^{1-b} < R < 2\}} R^{-(n+2)} |\nabla \tilde{\beta}|^2, 
\end{equation}
 \begin{equation}\label{eq: betaLowDegreeDecay}
 (2\theta)^{-n-4} \int_{\mfC \cap B_{2\theta} } |\tilde{\beta}_{(0,2)}|^2 \leq 10^{-2} \int_{\mfC \cap\{2^{1-b} < R < 2\}} R^{-(n+2)} |\nabla \tilde{\beta}|^2 
 \end{equation}
 and
 \begin{equation}\label{eq: yLowDegreeDecay}
 (2\theta)^{-n-2} \sum_{i=1}^{n}\int_{\mfC \cap B_{2\theta}} |(\tilde{y_i})_{(0,1)}|^2  \leq 10^{-2} \int_{\mfC \cap \{2^{1-b} < R < 2\}} R^{-(n+2)} |\nabla \tilde{\beta}|^2. 
 \end{equation}
Since $\tilde{\beta}_{>2}, \tilde{f}_{>2}, \tilde{y}_{>2}$ have degree strictly larger than $2$, we have the decay estimates
 \begin{equation}
 \begin{aligned}\label{eq: highDegreeDecay}
 (\theta)^{-n-4}\int_{B_{\theta}}(\tilde{f}_{>2})^2&\leq \theta^{2\alpha} C\\
 (2\theta)^{-n-4}\int_{B_{2\theta}}(\tilde{\beta}_{>2})^2  &\leq (2\theta)^{2\alpha}\int_{B_{1}}\tilde{\beta}^2\\
(2\theta)^{-n-2}\int_{B_{2\theta}}(\tilde{y}_i)_{>1}^2 &\leq (2\theta)^{2\alpha}\int_{B_1}\tilde{y}_i^2.
 \end{aligned}
 \end{equation}
 for $\alpha$ depending only on the spectrum of $\Sigma$.  
 
Let us address the terms $\tilde{f}_{2}, (\tilde{y}_{i})_{1}$, where we note that $(\tilde{y}_{i})_1 = -\frac{\del}{\del x_i} \tilde{f}_{2}$.  From the bounds of $\tilde{f}$ we evidently have 
\[
\sup_{\mfC \cap B_{\frac{3}{2}}} |\tilde{f}_2| \leq C
\]
for a constant $C$ depending only on $\cK, \tau$.

First consider the index set $I_2$.  Note that if $(j,p,\ell) \in I_2$ and $|\ell|=2$ then $p=\alpha_j=0$ the corresponding harmonic function is generated by $\mathfrak{su}(n)$. If $|\ell|=1$, then $2p-\alpha_j=1$ and hence $p=0, -\alpha_j=1$.  However, since $\Sigma$ is integrable in the sense of Definition~\ref{defn: strongInt}, any such harmonic function is generated by the action of $\mathfrak{su}(n)$ by Lemma~\ref{lem: rigid}.    Define
\[
q= \sum_{(j,p,0) \in I_{2}} \sum_{i=1}^{N_j} a^j_{p,\ell,i} r^{2p-\alpha_j}\phi^j_{i} 
\]
so $q$ is a quadratic growth harmonic function on $\bfC$ (note that in fact we must have $p=0$, since $q$ is harmonic).  Then $\tilde{f}_2-q$ is a harmonic function generated by the action of some element $\hat{a} \in \mathfrak{su}(n)$.  From the $L^{\infty}$ bounds for $\tilde{f}_2$ we have
\[
|\hat{a}|_{\mathfrak{su}(n)} + |q|_{L^2{(\Sigma)}} \leq C
\]
for a constant $C$ depending only on $\cK, \tau$.  By the integrability assumption on $\bfC$ and Lemma~\ref{lem: rigid} we can find sequences $\epsilon_{\nu}, \epsilon_{\nu}' \leq C \|\underline{f}_{\nu}\|_{L^{2}}(2,\tau) \rightarrow 0$ such that, with respect to the slightly modified data
\[
\begin{aligned}
\widehat{N}_\nu &:= e^{-\epsilon_\nu \hat{a}}\cdot N_\nu \\
 \bfC_{\epsilon_\nu'q}&: = \Cone(\exp^\cM_{\Sigma}\epsilon_\nu'q)\\
 \mfC_{\nu}&: = \mathbb{R}^k\times \bfC_{\epsilon_\nu'q}
 \end{aligned}
\]
 then we have
\begin{itemize}
\item  $\bfC_{\epsilon_\nu'q} \in \cK$.
\\
\item For $\nu$ sufficiently large, $\widehat{N}_{\nu}$ can be written as a graph of a function $\hat{f}_{\nu}$ satisfying $P_{1}(\eta, \tau, C\delta_{\nu})$, $P_2(\frac{5}{6})$ with respect to $\widehat{\mfC}_{\nu}$ as well as conditions $(ii), (iii)$, after possibly replacing $C_1, C_2$ with $2C_1, 2C_2$.  
\\
\item Passing to the limit we obtain a harmonic function $\tilde{f}$ on $\mfC\cap B_{\frac{3}{2}}$ with $\tilde{f}_{2}=0$.  
\end{itemize}
Thus, for $\theta$ sufficiently small depending only on $\cK$, and $b$ sufficiently large depending on $\theta, \cK$ and $\nu$ sufficiently large we have
\begin{itemize}
\item From ~\eqref{eq: betaLowDegreeDecay},~\eqref{eq: yLowDegreeDecay} and~\eqref{eq: highDegreeDecay} and the convergence in $L^2$ of $\tilde{\beta}_{\nu}, \tilde{y}_{i,\nu}$ to $\tilde{\beta}, \tilde{y}_{i}$ we obtain
\[
\begin{aligned}
\|\underline{\beta-Av_{\widehat{N}_{\nu}}(\beta, 2\theta)}\|^2_{L^{2}(\widehat{N}_{\nu}\cap B_{2\theta})} &\leq c(b){\rm VolEx}_{N_{\nu}}(2, 2^{1-b}) +10^{-5}\|\beta\|^2_{L^{2}(\widehat{N}_{\nu}\cap B_{1})}\\
\|\underline{y_i}\|^2_{L^{2}(\widehat{N}_{\nu}\cap B_{2\theta})}  &\leq c(b){\rm VolEx}_{N_{\nu}}(2, 2^{1-b}) + 10^{-5}\|\underline{y_i}\|^2_{L^{2}(\widehat{N}_{\nu}\cap B_{1})}
\end{aligned}
\]
Noting that $\|\underline{\beta}\|^2_{L^2(\widehat{N}_{\nu}\cap B_1)} \leq  4^{4}\|\underline{\beta}\|^2_{L^2(\widehat{N}_{\nu}\cap B_4)}$ (and similarly for $y_i$) yields conclusion ${\bf(c1)}$.
\\
\item From the bounds~\eqref{eq: flowDegreeDecay} and~\eqref{eq: highDegreeDecay}, for $\nu$ sufficiently large depending on $\theta$, $\widehat{N}_{\nu}$ is a graph over an open neighborhood of $\widehat{\mfC}_{\nu}\cap B_{\theta} \cap \{r>\theta \tau\}$ of a function $f_{\nu}$ satisfying $P_{1}(\eta, \tau, C\delta_{\nu})$, $P_2(\frac{5}{6})$ and such that
\[
\theta^{-n-4}\int_{\mfC_{\nu}\cap B_{\theta} \cap \{r>\theta \tau\}} f_{\nu}^2 \leq c(b){\rm VolEx}_{N_{\nu}}(2,2^{1-b}) + 10^{-1}\int_{\mfC_{\nu}\cap B_2 \cap\{ r> 2\tau\}} f_{\nu}^2
\]
\item From the estimates~\eqref{eq: ellipticEstimates} we can choose $\nu$ sufficiently large depending only on $\cK, \eta, \tau$ so that
\[
\sup_{\hat{\mfC} \cap B_{\theta} \cap \{r> \theta\tau\}} r^{-1}|df_{\nu}| + |D^2f_{\nu}| \leq \eta.
\]
establishing conclusion ${\bf (c2)}$.
\end{itemize}

\end{proof}

We now come to the proof of the main theorem.

\begin{proof}[Proof of Theorem~\ref{thm: main1}]
Suppose $N$ is an exact special Lagrangian in $B_{100}$ such that $\lambda_\ell N=N_{\ell} \rightarrow \mfC = \mathbb{R}^k\times \Cone(\Sigma)$ for some sequence $\lambda_\ell$ increasing to $+\infty$. Let $\mathcal{K}\subset \mathcal{M}$ be a connected compact set with non-empty interior parametrizing smooth special Legendrian deformations of $\Sigma$; we may assume $\mathcal{K}$ can be identified with a small closed ball in $T_{\Sigma}\mathcal{K}$ by the exponential map $\exp^{\cM}_{\Sigma}$. 

Given $\tau_0$ (to be fixed below) we fix $\eta >0$ so that
\[
\eta < \min\{ \eta_2(\mathcal{K}), \eta_3(\mathcal{K}, \frac{5}{6}), \eta_4 (\mathcal{K}, \tau_0), \eta_5(\mathcal{K}) \}.
\]
Let $\epsilon >0$ be given.  By Allard compactness, for $\ell$ sufficiently large depending on $\eta, \epsilon$ and $\cK$ sufficiently small we may assume that $N_{\ell}$ satisfies
\begin{itemize}
\item[$(i)$] The volume property $P_2(\frac{5}{6})$ at all scales $s\leq1$, and with respect to all cylinders $\mathbb{R}^k\times \Cone(\Sigma_{\kappa})$ for $\kappa \in \cK$.
\item[$(ii)$] We can choose $\tau_0$ sufficiently small so that
\[
\rho^{-n}\mathcal{H}^{n}(\exp(-a)N_{\ell} \cap B_{\rho} \cap \{r<4\rho\tau_0\}) < \omega_n2^{-(n+8)} 10^{-6}
\]
for all $\rho <4$ and for all $a \in \mathfrak{su}(n)$ with $|a|_{\mathfrak{su(n)}}\leq 1$.  
\item [$(iii)$] The small graph property $P_1(\eta, \tau_0, \epsilon)$ with respect to $\mfC$.
\item[$(iv)$] The harmonic property $P_3(\epsilon)$.
\item[$(v)$] The volume excess satisfies ${\rm VolEx}_{N}(2)<\epsilon$.
\end{itemize}

Let $b=b_0(\cK, \tau_0) \in \mathbb{Z}_{\geq 0}$ and $\theta= \theta(\cK)$ be the constants appearing in Proposition~\ref{prop: fDecay}. Let $N= N_{\ell}$ to ease notation.  We make the following induction statement.
\\

{\bf Induction statement, $I(s_0)$:} For all $s\in \mathbb{Z}_{\geq 0}$ and $s\leq s_0$ there is a sequence of scales $\rho_s= (\frac{1}{4})^{s_1}(\frac{\theta}{2})^{s_2}$ satisfying
\[
\rho_{s}= \frac{1}{4}\rho_{s-1} \quad \text{ or }\quad \rho_s = \frac{\theta}{2}\rho_{s-1}
\]
and a constant $\uC$ independent of $s$ such that, if we define
\[
\phi(s) = 10^{-s}\epsilon +\uC\sum_{\ell=0}^{s-1}10^{\ell-s}{\rm VolEx}_{N}(2\rho_{\ell}, 2^{1-b}\rho_{\ell})
\]
then
\begin{itemize}
\item[(1)] There is an exact special Lagrangian $N_{s}$ and a special Lagrangian cylinder $\mfC_{s} = \mathbb{R}^k\times \Cone(\exp_{\Sigma_{s-1}}^{\cM}q_s)$ defined inductively by $N_0= N, \Sigma_{0}=\Sigma$, and
\[
N_{s} = \exp(-a_{s})N_{s-1}, \qquad \mfC_{s} = \mathbb{R}^k\times \Cone(\exp_{\Sigma_{s-1}}^{\cM}q_s)
\]
for $a_s \in \mathfrak{su}(n)$, $q_{s}\in \mfH_{2m}(\Sigma_{s-1})$ satisfying
\[
|a_s|_{\mathfrak{su}(n)} + |q_{s}|_{L^2(\Sigma_{s-1})} \leq \uC^2\phi(s-1).
\]
\item[(2)] Define $\beta_{s} = \beta - Av_{N_s}(\beta, 4\rho_s)$.  With respect to this normalization, $N_s$ is the graph associated to a normalized potential $f_s$ defined over $\mfC_{s}\cap B_{2\rho_s}\cap\{r>2\rho_s\tau_0\}$ and we have the estimates
\begin{equation}\label{eq: inductiveEstBetaY}
\|\underline{\beta}_s, \underline{y}\|^2_{L^2(N\cap B_{4\rho_s})} \leq \phi(s)
\end{equation}
\begin{equation}\label{eq: inductiveEstfL2}
\|\underline{f}_s\|^2_{L^2}(2\rho_s, \tau_0) \leq \uC\phi(s)
\end{equation}
\begin{equation}\label{eq: inductiveEstdf}
\sup_{\mfC_{s}\cap B_{2\rho_s} \cap \{r>2\rho_s\tau_0\}} r^{-1}|df_s| +|D^2f_{s}| < \eta.
\end{equation}
In particular, with respect to this normalization, and at scale $\rho_s$, $N_{s}$ satisfies:
\begin{itemize}
\item the small graph property $P_1(\eta, \tau_0, \uC\phi(s))$ with respect to $\mfC_{s}$,
\item the volume property $P_2(\frac{5}{6})$ with respect to $\mfC_{s}$, 
\item the harmonic property $P_3(\phi(s))$
\end{itemize}
\end{itemize}

We will show that $I(s-1)$ implies $I(s)$ if $\uC$ is chosen large depending on $\cK,\tau_0$, and $\epsilon$ is chosen small depending only on $\cK, \tau_0, \eta, \uC$. Note that from the properties of $N$ (c.f. property (v)) and since $b<+\infty$ we have
\[
\phi(s) \leq 2\uC\epsilon, \qquad  \sum_{s=0}^{s_0} \phi(s) \leq C(\uC,\theta, b) \epsilon.
\]
Thus, as long as $I(s)$ holds, the special Legendrians $\Sigma_{s} \in \cK$, provided $\epsilon$ is chosen sufficiently small depending on $\uC, \cK, \tau_0$. The base case, $I(0)$, follows from our choice of $N$ as above.  We now consider a trichotomy of cases, which we separate into lemmas.  

\begin{lem}\label{lem: case1}
In the above setting, suppose that $I(s-1)$ holds and
\[
\|\underline{f}_{s-1}\|(2\rho_{s-1}, \tau_0) \leq 10^{-2}\delta_3\tau_0^2 \|\underline{\beta}_{s-1},\underline{y}\|_{L^{2}(N\cap B_{4\rho_{s-1}})}
\]
where $\delta_3=\delta_3(\cK)$ is the constant appearing in Lemma~\ref{lem: easyDecay}.  Then, $I(s)$ holds provided $\uC$ is chosen sufficiently large depending only on $\cK, \tau_0$ and $\epsilon$ is sufficiently small depending on $\cK, \tau_0$.
\end{lem}

\begin{proof}
By the induction assumption and our choice of $\tau_0$ (c.f. property (ii) above) we may apply Lemma~\ref{lem: easyDecay} to deduce that at scale $\rho_{s} =\frac{1}{4}\rho_{s-1}$, we have
\[
\|\underline{\beta}_{s-1}, \underline{y}\|^2_{L^{2}(N_{s-1}\cap B_{4\rho_{s}})} \leq \frac{1}{100}\|\underline{\beta}_{s-1}, \underline{y}\|^2_{L^{2}(N_{s-1}\cap B_{4\rho_{s-1})}}
\]
Set $N_{s}=N_{s-1}$ and $\mfC_{s}= \mfC_{s-1}$.  Then statement $(1)$ of $I(s)$ is trivially satisfied.  Since subtracting the average decreases the $L^2$ norm, $\beta_{s} = \beta_{s-1}-Av_{N_{s}}(\beta_{s-1}, 4\rho_s)$ has
\[
\|\underline{\beta}_{s}, \underline{y}\|^2_{L^{2}(N_{s}\cap B_{4\rho_{s}})} \leq \frac{1}{100}\|\underline{\beta}_{s-1}, \underline{y}\|^2_{L^{2}(N_{s-1}\cap B_{4\rho_{s-1})}}
\]
which easily implies~\eqref{eq: inductiveEstBetaY}. 

We now apply Proposition~\ref{prop: propOfSmall} with $\gamma =\frac{5}{6}$ to conclude that if $\epsilon$ is chosen sufficiently small depending on $\cK, \tau_0$, then $f_{s-1}$ extends to $B_{2\rho_{s}}\cap \{r>2\rho_s\tau_0\}$ and satisfies
\[
\begin{aligned}
\|\underline{f}_{s-1}\|^2_{L^{2}}(2\rho_s, \tau_0) &\leq C(\|\underline{\beta}_{s-1}, \underline{y}\|^2_{L^{2}(N_{s-1}\cap B_{4\rho_{s-1}})} + \|\underline{f}_{s-1}\|_{L^2}^2(2\rho_{s-1}, \tau_0) \\
&\leq 2C\|\underline{\beta}_{s-1}, \underline{y}\|^2_{L^{2}(N_{s-1}\cap B_{4\rho_{s-1}})} 
\end{aligned}
\]
as well as
\[
\begin{aligned}
\sup_{\mfC_{s}\cap B_{2\rho_{s}}\cap\{r>2\rho_s\tau_0\}} r^{-1}|df_{s-1}| + |D^2 f_{s-1}| \leq 2C\|\underline{\beta}_{s-1}, \underline{y}\|_{L^{2}(N_{s-1}\cap B_{4\rho_{s-1}})} 
\end{aligned}
\]
for a constant $C$ depending only on $\cK, \tau_0$.  Now the normalized potential is given by 
\[
f_{s} = f_{s-1}-Av_{N_{s}}(\beta_{s-1}, 4\rho_{s}).
\]
Since $I(s-1)$ holds, we obtain~\eqref{eq: inductiveEstdf} for $\epsilon$ small depending on $\eta, \cK, \tau_0$.  By the mean-value inequality we have
\begin{equation}\label{eq: controlOfAverages}
(2\rho_{s})^{-n-4} \int_{\mfC_{s}\cap B_{2\rho_{s}}\cap \{r>2\rho_{s}\tau_0\}}(Av_{N_{s}}(\beta_{s-1}, 4\rho_{s}))^2 \leq C\|\underline{\beta}_{s-1}\|^2_{L^{2}(N_{s-1}\cap B_{4\rho_{s-1}})}
\end{equation}
for a constant $C$ depending on $\cK, \tau_0$.  All together we have
\[
\|\underline{f}_{s}\|^2_{L^{2}}(2\rho_s, \tau_0) \leq C\|\underline{\beta}_{s-1}, \underline{y}\|^2_{L^{2}(N_{s-1}\cap B_{4\rho_{s-1}})} 
\]
for $C$ depending only on $\cK, \tau_0$.  Provided the constant $\uC$ is chosen large, depending only on $\cK, \tau_0$, this implies~\eqref{eq: inductiveEstfL2}.    Taken together these estimates imply property (2), and hence $I(s)$ holds.
\end{proof}

Next we consider the case when the volume excess is dominant.

\begin{lem}\label{lem: case2}
In the above setting, suppose that $I(s-1)$ holds and
\[
\begin{aligned}
10^{-2}\delta_3\tau_0^2 \|\underline{\beta}_{s-1},\underline{y}\|_{L^{2}(N\cap B_{4\rho_{s-1}})}&\leq \|\underline{f}_{s-1}\|(2\rho_{s-1}, \tau_0)\\ 
\|\underline{\beta}_{s-1}, \underline{y}_{s-1}\|^2_{L^{2}(N_{s-1}\cap B_{4\rho_{s-1}})} + \|\underline{f}_{s-1}\|^2_{L^{2}(2\rho_{s-1}, \tau_0)} &\leq {\rm VolEx}_{N}(2\rho_{s-1},2^{1-b}\rho_{s-1})
\end{aligned}
\]
Then, $I(s)$ holds provided $\uC$ is chosen sufficiently large depending on $\cK, \tau_0$ and $\epsilon$ is sufficiently small depending on $\uC, \cK, \tau_0$.
\end{lem}

\begin{proof}
Set $\rho_{s} = \frac{1}{4}\rho_{s-1}$ and let $N_{s}= N_{s-1} =\widehat{N}$ and $\mfC_{s} = \mfC_{s-1} =\widehat{\mfC}$, where we briefly suppress the dependence on $s$ to ease the notation.  Then part (1) of $I(s)$ is trivially satisfied.  From the mean-value inequality for subharmonic functions we have
\[
\|\underline{\beta_{s-1}}\|^2_{L^{2}(\widehat{N}\cap B_{4\rho_{s}})} \leq 4^{4} \|\underline{\beta_{s-1}}\|^2_{L^{2}(\hatN\cap B_{4\rho_{s-1}})}
\]
Now since $\beta_{s}= \beta_{s-1}-Av_{\hatN}(\beta_{s-1}, 4\rho_{s})$, and subtracting the average decreases the $L^2$ norm the assumption of the Lemma yields
\[
\|\underline{\beta_{s}}\|^2_{L^{2}(\hatN\cap B_{4\rho_{s}})}\leq 4^4{\rm VolEx}_{N}(2\rho_{s-1},2^{1-b}\rho_{s-1}).
\]
Thus \eqref{eq: inductiveEstBetaY} holds provided $\underline{C}$ is chosen large.  Now by volume monotonicity we have ${\rm VolEx}_{N}(2\rho_{s-1},2^{1-b}\rho_{s-1}) \leq \epsilon$.  By Proposition~\ref{prop: propOfSmall} we conclude that if $\epsilon$ is sufficiently small depending on $\cK, \tau_0$ then $f_{s-1}$ extends to\\ $\hatmfC \cap B_{2\rho_s} \cap\{r>2\rho_{s}\tau_0\}$ and satisfies
\[
\|\underline{f}_{s-1}\|^2_{L^{2}}(2\rho_s, \tau_0) \leq C{\rm VolEx}_{N}(2\rho_{s-1},2^{1-b}\rho_{s-1}),\\
\]
and
\[
\sup_{\hatmfC \cap B_{2\rho_s} \cap\{r>2\rho_{s}\tau_0\}} r^{-1}|df_{s-1}|+ |D^2f_{s-1}| \leq C\left({\rm VolEx}_{N}(2\rho_{s-1},2^{1-b}\rho_{s-1})\right)^{\frac{1}{2}}
\]
where we used the assumptions of the Lemma, and $C$ is a constant depending on $\cK, \tau_0$.  Now since $f_{s} = f_{s-1}-Av_{\hatN}(\beta_{s-1}, 4\rho_{s})$ this implies~\eqref{eq: inductiveEstdf} provided $\epsilon$ is small depending on $\cK, \tau_0, \eta$.  To obtain ~\eqref{eq: inductiveEstfL2} we can apply~\eqref{eq: controlOfAverages} to conclude
\[
 \|\underline{f}_{s}\|^2_{L^{2}}(2\rho_s, \tau_0) \leq C{\rm VolEx}_{N}(2\rho_{s-1},2^{1-b}\rho_{s-1})
 \]
for $C$ depending on $\cK, \tau_0$.  Thus, provided $\uC$ is large depending on $\cK, \tau_0$, and $\epsilon$ is small depending on $\uC, \cK, \tau_0$, we conclude the $I(s)$ holds.
\end{proof}

Finally, we come to the main case of the induction

\begin{lem}\label{lem: case3}
In the above setting, suppose that $I(s-1)$ holds and
\[
\begin{aligned}
\|\underline{\beta}_{s-1},\underline{y}\|_{L^{2}(N\cap B_{4\rho_{s-1}})} &\leq \left(10^{-2}\delta_3\tau_0^2\right)^{-1} \|\underline{f}_{s-1}\|(2\rho_{s-1}, \tau_0)\\
 {\rm VolEx}_{N}(2\rho_{s-1},2^{1-b}\rho_{s-1}) &\leq \|\underline{\beta}_{s-1}, \underline{y}_{s-1}\|^2_{L^{2}(N_{s-1}\cap B_{4\rho_{s-1}})} + \|\underline{f}_{s-1}\|^2_{L^{2}(2\rho_{s-1}, \tau_0)} \end{aligned}
\]
Then, $I(s)$ holds provided $\uC$ is chosen sufficiently large depending on $\cK, \tau_0$, and $\epsilon$ is sufficiently small depending on $\uC, \cK, \tau_0$.
\end{lem}
\begin{proof}
Since $I(s-1)$ holds, we can apply Proposition~\ref{prop: fDecay} provided $\epsilon$ is chosen small depending on $\cK, \delta_{3}(\cK), \tau_0, \uC, \eta$.  Here we use that
\[
 {\rm VolEx}_{N}(2\rho_{s-1},2^{1-b}\rho_{s-1}) = {\rm VolEx}_{N_{s-1}}(2\rho_{s-1},2^{1-b}\rho_{s-1}) 
\]
by $SU(n)$ invariance.  Let $\rho_{s} = \left(\frac{\theta}{2}\right)\rho_{s-1}$.  Proposition~\ref{prop: fDecay} yields the following conclusions
\begin{itemize}
\item there exist special Lagrangians $N_{s}$ and $\mfC_{s}$ satisfying part (1) of $I(s)$, provided $\uC$ is chosen sufficiently large depending on $\cK, \tau_0$.
\item we have the estimate 
\[
\|\underline{\beta}_{s},\underline{y}\|^2_{L^{2}(N_{s}\cap B_{4\rho_{s}})} \leq  {\rm VolEx}_{N}(2\rho_{s-1},2^{1-b}\rho_{s-1}) +\frac{1}{10}\|\underline{\beta}_{s-1}, \underline{y}\|^2_{L^2(N_{s-1}\cap B_{4\rho_{s-1}})}
\]
which easily implies~\eqref{eq: inductiveEstBetaY}.
\item $N_{s}$ is the graph associated to a normalized potential function $f_{s}$ defined on $\mfC_{s}\cap B_{2\rho_{s}}\cap \{r>2\rho_s\tau_0\} $ satisfying
\[
\|\underline{f}_s\|^2_{L^{2}}(2\rho_{s},\tau_0) \leq {\rm VolEx}_{N}(2\rho_{s-1},2^{1-b}\rho_{s-1}) + \frac{1}{10} \|\underline{f}_{s-1}\|^2_{L^{2}}(2\rho_{s-1},\tau_0)
\]
as well as
\[
\sup_{\mfC_{s}\cap B_{2\rho_{s}}\cap \{r>2\rho_s\tau_0\}} r^{-1}|df_{s}| + |D^2f_{s}| < \eta
\]
which yields ~\eqref{eq: inductiveEstfL2} and~\eqref{eq: inductiveEstdf}
\end{itemize}

\end{proof}

We can now finish the proof.  Since one of Lemmas~\ref{lem: case1},~\ref{lem: case2}, or~\ref{lem: case3} must hold at each scale, we conclude that the induction statement $I(s)$ holds for all $s \in \mathbb{Z}\geq 0$, as long as $\uC$ is chosen large depending on $\cK, \tau_0$ and $\epsilon$ is chosen sufficiently small depending on $\uC, \cK, \tau_0$.  We note that for any $m\in \mathbb{Z}_{\geq 0}$ we have the summability of error bounds
\[
\Phi(m):= \sum_{s=m}^{\infty}\phi(s) \leq C\left(10^{-m}\epsilon + {\rm VolEx}_{N}(2\rho_{m-1})\right)
\]
for a constant $C$ depending on $\uC, \theta, b$.  By part $(1)$ of $I(s)$ we conclude
\begin{itemize}
\item For each $s$ there is an element $\hat{a}_{s}\in \mathfrak{su}(n)$ so that $N_{s}= \exp(-\hat{a}_{s})N$.
\\
\item There is an element $\hat{a} \in \mathfrak{su}(n)$ so that,
\begin{equation}\label{eq: finalAbnd}
\|\hat{a}\|_{\mathfrak{su}(n)} \leq C\epsilon, \quad \text{ and, for all $m$} \quad |\hat{a}_{m} - \hat{a}|_{\mathfrak{su}(n)} \leq C\Phi(m).
\end{equation}
\item There is a special Lagrangian cylinder $\hatmfC = \mathbb{R}^k \times \Cone(\widehat{\Sigma})$, where $\cK \ni \widehat{\Sigma} = \exp^{\cM}_{\Sigma}q$ for $q\in \mfH_{2m}(\Sigma)$ satisfying 
\begin{equation}\label{eq: finalqbnd}
\|q\|_{L^{2}(\Sigma)} \leq C\epsilon.
\end{equation}
\item We have a Hausdorff distance estimate in $B_{2}$
\[
d^{H}(\exp(\hat{a}_m)\mfC_{m}, \exp(\hat{a})\hatmfC \,; B_2) \leq C\Phi(m).
\]
\item By part (2) of $I(s)$ and Proposition~\ref{prop: propOfSmall}, for any $m\in \mathbb{Z}_{\geq 0}$, $N_{m}$ can be written as the graph of a function $f_m$ defined on $\hatmfC \cap B_{2\rho_{m}}\cap \{r >2\rho_{m}\tau_m\}$ for
\[
\tau_{m} = \Phi(m)^{\frac{1}{8}}
\]
 satisfying
\begin{equation}\label{eq: finalHigherOrderEst}
\sup_{\hatmfC\cap B_{2\rho_{m}} \cap \{r>2\rho_m\tau_m\}} r^{-1}|df_{m}| + |D^2f_{m}| \leq C\Phi(m)^{\frac{1}{4}}.
\end{equation}
\item The estimate~\eqref{eq: finalHigherOrderEst} implies that every tangent cone of $N$ agrees with $\exp(\hat{a})\hatmfC$ on $B_{2}\cap \{r> \delta\}$ for any $\delta>0$. On the other hand, since $\epsilon$ may be chosen arbitrarily small in ~\eqref{eq: finalAbnd} and~\eqref{eq: finalqbnd}, this implies $\hat{a}=0$ and $\hatmfC=\mfC$.  By volume monotonicity this is already sufficient to yield the uniqueness of the tangent cone.
\\
\item Furthermore, since $|df_m|$ controls the distance from $N_{m}$ to $\mfC$, ~\eqref{eq: finalHigherOrderEst} implies a Hausdorff distance bound, 
\[
d^{H}(N_{m}, \mfC;\, B_{2\rho_{m}} \cap \{r>\rho_m\tau_m\}) \leq C\rho_{m}\Phi(m)^{\frac{1}{4}}
\]
\item From Corollary~\ref{cor: easyVolEst} we have a volume bound
\[
\bigg| \frac{\mathcal{H}^{n}(N_{m} \cap B_{\rho_{m}} \cap \{r>\rho_m\tau_m\})}{\rho_m^n} - \frac{\mathcal{H}^{n}(\mfC \cap B_{\rho_{m}} \cap \{r>\rho_m\tau_m\})}{\rho_m^{n}}\bigg| \leq C\Phi(m)^{\frac{1}{4}}.
\]
\\
\end{itemize}

All together, these considerations together with Lemma~\ref{lem: HausdorffBound} below imply the Hausdorff distance bound
\[
d^{H}(N, \mfC;\, B_{\rho_m}) \leq C\rho_{m}\left(\Phi(m)^{\frac{1}{4n}} + {\rm VolEx}_{N}(2\rho_m)^{\frac{1}{n}}\right).
\]
It is not hard to show that this bound implies the rate estimate
\begin{equation}\label{eq: main1RateEst}
d^{H}(N, \mfC;\, B_{\rho}) \leq C\rho\left(\rho^{\alpha} + {\rm VolEx}_{N}(\rho^{\alpha})^{\frac{1}{4n}}\right)
\end{equation}
for some $\alpha>0$.
\end{proof}

In the proof of Theorem~\ref{thm: main1} we used the following simple result.
\begin{lem}\label{lem: HausdorffBound}
Suppose $N\subset \mathbb{C}^n$ is a connected special Lagrangian in $B_{2\rho}$ with $0\in N$, and suppose we have a special Lagrangian cylinder of the form $\mfC= \mathbb{R}^k\times \Cone(\Sigma)\subset \mathbb{C}^n$.  Suppose in addition that there are constants $\epsilon, \delta >0$ and $\tau\in (0,2)$ such that we have the following estimates:
\[
\rho^{-n}|\cH^{n}(N\cap B_{\rho}\cap\{r>\rho\tau\}) - \cH^{n}(\mfC\cap B_{\rho} \cap \{r>\rho\tau\})| \leq \delta
\]
\[
d^{H}(N,\mfC; B_{\rho}\cap\{r>\rho\tau\}) \leq \rho\epsilon.
\]
Then we have the Hausdorff distance bound
\[
d^{H}(N, \mfC; B_{\rho}) \leq \rho\epsilon + \rho C\left(\rho^{-n}\left[\cH^{n}(N\cap B_{\rho}) -\cH^{n}(\mfC\cap B_{\rho})\right]) + \tau^{n-k} +\delta\right)^{\frac{1}{n}}
\]
for a constant $C$ depending only on $\mfC ,n$.
\end{lem}
\begin{proof}
The proof is straightforward.  By rescaling we may assume $\rho=1$.  Suppose there exists a point $p\in N$ such that ${\rm dist}(p, \mfC)= \gamma > \epsilon$. Since $N$ is connected we assume that ${\rm dist}(p, \del B_{1}) >\gamma$.  Consider the ball $B_{\gamma-\epsilon}(p) \subset B_{1}$.  Clearly $B_{\gamma-\epsilon}(p) $ is disjoint from $N\cap B_{1}\cap\{r>\tau\}$.  Thus, from volume monotonicity we have
\[
\begin{aligned}
\cH^{n}(N\cap B_{1}) &\geq \omega_{n}(\gamma-\epsilon)^{n} + \cH^{n}(N\cap B_{1}\cap \{r>\tau\})\\
&\geq \omega_{n}(\gamma-\epsilon)^{n}+ \cH^{n}(\mfC \cap B_{1}) - C\tau^{n-k}-\delta
\end{aligned}
\]
for a constant $C$ depending only on $\mfC$.  Reorganizing yields the result.
\end{proof}

Next we prove Theorem~\ref{thm: main2}.

\begin{proof}[Proof of Theorem~\ref{thm: main2}]
The proof is similar to the proof of Theorem~\ref{thm: main1}, making use of the volume excess decay estimate in Proposition~\ref{prop: decayOfVolEx}, and exploiting the $b=+\infty$ case of Proposition~\ref{prop: fDecay}. We only sketch the argument. We modify the induction statement $I(s_0)$ as follows:
\begin{itemize}
\item Replace the rate function $\phi$ with the power law decay
\[
\phi(s)= 10^{-s}\epsilon
\]
\item Clause $(1)$ of the induction statement remains the same.
\item To clause $(2)$ we add the estimate
\begin{equation}\label{eq: inductionVolDecay}
{\rm VolEx}_{N_s}(2\rho_{s}) \leq \underline{C}^2\phi(s)
\end{equation}
\end{itemize}

To simplify the exposition, let us denote by $\beta_{s} = \beta - Av_{N_{s}}(\beta, 4\rho_s)$ and
\[
\begin{aligned}
\mathcal{B}_s&:= \|\underline{\beta}_{s}, \underline{y}\|^2_{L^2(N_{s}\cap B_{4\rho_s})}\\
\mathcal{F}_s &:= \|\underline{f}_s\|^2_{L^{2}}(2\rho_s,\tau_0)\\
\mathcal{V}_{s} &:= {\rm VolEx}_{N_{s}}(2\rho_s)
\end{aligned}
\]
where $f_s$ is the normalized potential relative to our choice of $\beta_s$. In case the assumptions of Lemma~\ref{lem: case1} hold we obtain from the same argument
\[
\cB_{s} \leq \frac{1}{100} \cB_{s-1}, \quad \text{ and } \quad \cF_{s} \leq C\cB_{s-1}.
\]
for $C$ depending on $\cK, \tau_0$.  This yields the desired estimates for $\uC$ sufficiently large depending on $\cK, \tau_0$. On the other hand, from Proposition~\ref{prop: decayOfVolEx} we have
\[
\cV_{s} \leq C(\cF_{s-1}+\cB_{s-1}+ \cV_{s-1}^{\frac{n}{n-1}})
\]
for a constant $C$ depending on $\cK$. Now by the assumption of Lemma~\ref{lem: case1} we have $\cF_{s-1} \leq \hat{\delta} \cB_{s-1}$ for $\hat{\delta} \ll1$.  Thus we get
\[
\cV_{s} \leq 2C \cB_{s-1} +C\cV_{s-1}^{\frac{n}{n-1}} \leq 20C \phi(s) + CV_{s-1}^{\frac{1}{n-1}} \underline{C}\phi(s)
\]
Now provided $\epsilon$ is sufficiently small depending on $\uC$, and $\uC$ is sufficiently large depending on $C$ (and hence on $\cK$) we have $\cV_{s} \leq \uC^2\phi(s)$ as desired.

We replace Lemma~\ref{lem: case2} with a slightly modified statement

\begin{lem}\label{lem: case2'}
Suppose that $I(s-1)$ holds and that
\[
\begin{aligned}
\hat{\delta} \cB_{s-1} &\leq \cF_{s-1}\\
\cB_{s-1} +\cF_{s-1} &\leq C_{2}^{-1}\cV_{s-1}
\end{aligned}
\]
Then $I(s)$ holds provided $C_{2}=10^4 \uC^2$ and  $\uC$ is chosen sufficiently large depending on $\cK, \tau_0$, and $\epsilon$ is chosen sufficiently small depending on $\uC, \cK, \tau_0$.
\end{lem}
\begin{proof}
From Proposition~\ref{prop: decayOfVolEx} we get
\[
\cV_{s} \leq C(\cF_{s-1}+\cB_{s-1} + \cV_{s-1}^{\frac{n}{n-1}}) \leq C\cV_{s-1}(C_{2}^{-1} + \cV_{s-1}^{\frac{1}{n-1}})
\]
for $C$ depending only on $\cK$.  Choosing $C_2=10^4 \uC^2$, and $\uC$ large depending on $\cK$ and $\epsilon$ small depending on $\uC, \cK$ we can arrange
\[
C\cV_{s-1}(C_{2}^{-1} + \cV_{s-1}^{\frac{1}{n-1}}) \leq \frac{1}{10}\cV_{s-1}
\]
yielding the desired decay.  The remainder of the argument proceeds in the same way as in Lemma~\ref{lem: case2} to obtain
\[
\cB_{s} \leq 4^{4}C_{2}^{-1} \cV_{s-1} \leq \phi(s)
\]
by our choice of $C_{2}$.  Finally, by Proposition~\ref{prop: propOfSmall} we have
\[
\cF_{s} \leq C(\cF_{s-1}+\cB_{s-1}) \leq \frac{C}{C_{2}}\cV_{s-1} 
\]
for $C$ depending only on $\cK, \tau_0$. Now by the induction assumption and our choice of $C_2$ we have $\cF_{s} \leq \frac{C}{10^3} \phi(s)$ and hence the desired conclusion holds provided $\uC$ is chosen large depending on $\cK, \tau_0$.
\end{proof}

Finally, we replace Lemma~\ref{lem: case3} with
\begin{lem}\label{lem: case3'}
Suppose that $I(s-1)$ holds and
\[
\begin{aligned}
\cB_{s-1} &\leq\hat{\delta}^{-1}\cF_{s-1}\\
 \cV_{s-1} &\leq C_{2}(\cB_{s-1} + \cF_{s-1})
 \end{aligned}
\]
Then, $I(s)$ holds provided $C_{2}=10^{4}\uC^2$ and $\uC$ is chosen sufficiently large depending on $\cK, \tau_0$, and $\epsilon$ is sufficiently small depending on $\uC, \cK, \tau_0$.
\end{lem}
\begin{proof}
The $b=+\infty$ case of Proposition~\ref{prop: fDecay} implies
\[
\cB_{s} \leq \frac{1}{10} \cB_{s-1} \quad \text{ and } \quad \cF_{s} \leq \frac{1}{10} \cF_{s-1}
\]
thus it suffices to establish the decay of the volume excess.  Since $\rho_{s} = \frac{\theta}{2} \rho_{s-1}$ in this case and $\theta < \frac{1}{2}$ we have
\[
\cV_s \leq {\rm VolEx}_{N}(\frac{1}{2}\rho_{s-1})
\]
and by Proposition~\ref{prop: decayOfVolEx} we have
\[
{\rm VolEx}_{N}(\frac{1}{2}\rho_{s-1}) \leq C\left( \cF_{s-1} +\cB_{s-1} + \cV_{s-1}^{\frac{n}{n-1}}\right).
\]
From the assumptions of the lemma we obtain
\[
\cV_{s} \leq C(1+ C_2\cV_{s-1}^{\frac{1}{n-1}})(\cF_{s-1}+\cB_{s-1}) \leq C(1+ C_2\cV_{s-1}^{\frac{1}{n-1}})(\uC+1)\phi(s-1).
\]
Choosing $\uC$ is large depending on $\cK, \tau_0$ and then $\epsilon$ small depending on $\uC$ we can arrange that $C_2\cV_{s-1}^{\frac{1}{n-1}} \leq 1$, and $\cV_{s} \leq \frac{1}{10}\uC^2\phi(s-1)$ which is the desired conclusion.
\end{proof}

It follows that $I(s)$ holds for all $s$ and from the improved rate estimate and the arguments in the proof of Theorem~\ref{thm: main1} we obtain the desired polynomial convergence.
\end{proof}

\end{document}